\documentclass{amsart} 
\usepackage{amssymb,amsmath,latexsym,times,tikz,hyperref,mathrsfs}

\hypersetup{colorlinks=true, linkcolor=blue, citecolor=magenta}

\newcommand{\btu}{\bigtriangleup}

\numberwithin{equation}{section}

\theoremstyle{plain}
\newtheorem{thm}{Theorem}[section]
\newtheorem{dfn}[thm]{Definition}
\newtheorem{cor}[thm]{Corollary}

\newtheorem{lemma}[thm]{Lemma}

\newtheorem{sublemma}{}[thm]

\theoremstyle{definition}
\newtheorem{example}{Example}

\newtheorem*{ex:bpm2}{Example \ref{ex:booleanp} (continued)} 
\newtheorem*{ex:aff2}{Example \ref{ex:affine} (continued)} 
\newtheorem*{ex:ppl2}{Example \ref{ex:projpl} (continued)} 
\newtheorem*{ex:matr2}{Example \ref{ex:matroid} (continued)} 

\newcommand{\del}{\backslash}
\DeclareMathOperator{\cl}{cl}

\title[Decomposable polymatroids]{Decomposable polymatroids and
  connections with graph coloring} 
\author[J.~Bonin]{Joseph E.~Bonin} \address[J.~Bonin]
{Department of Mathematics\\ The George Washington University\\
  Washington, D.C. 20052, USA} \email[J.~Bonin] {jbonin@gwu.edu}
\author[C.~Chun] {Carolyn Chun} \address[C.~Chun]
{Department of Mathematics\\
  United States Naval Academy\\
  Annapolis, MD, 21402, USA} \email[C.~Chun] {chun@usna.edu}
\date{\today}

\begin{document}

\begin{abstract}
  We introduce ideas that complement the many known connections
  between polymatroids and graph coloring.  Given a hypergraph that
  satisfies certain conditions, we construct polymatroids, given as
  rank functions, that can be written as sums of rank functions of
  matroids, and for which the minimum number of matroids required in
  such sums is the chromatic number of the line graph of the
  hypergraph.  This result motivates introducing chromatic numbers and
  chromatic polynomials for polymatroids.  We show that the chromatic
  polynomial of any $2$-polymatroid is a rational multiple of the
  chromatic polynomial of some graph.  We also find the excluded
  minors for the minor-closed class of polymatroids that can be
  written as sums of rank functions of matroids that form a chain of
  quotients.
\end{abstract}

\maketitle

\section{Introduction}\label{sec:intro}

The polymatroids that are of interest in this paper are often called
\emph{integer polymatroids} or \emph{discrete polymatroids} to
distinguish them from more general polymatroids.
Thus, here, a
\emph{polymatroid} on a finite set $E$ is a function
$\rho: 2^E\to \mathbb{Z}$ that has the following properties:
\begin{enumerate}
\item $\rho$ is \emph{normalized}, that is, $\rho(\emptyset) = 0$,
\item $\rho$ is \emph{non-decreasing}, that is, if
  $A\subseteq B\subseteq E$, then $\rho(A)\leq \rho(B)$, and
\item $\rho$ is \emph{submodular}, that is,
  $\rho(A\cup B)+\rho(A\cap B)\leq \rho(A) + \rho(B)$ for all
  $A, B\subseteq E$.
\end{enumerate}
Herzog and Hibi \cite{HerzogHibi} give some equivalent formulations of
this notion.

For a positive integer $k$, a \emph{$k$-polymatroid} is a polymatroid
$\rho$ on $E$ for which $\rho(e)\leq k$ for all $e\in E$, or,
equivalently, $\rho(A)\leq k|A|$ for all $A\subseteq E$.  Thus,
matroids are $1$-polymatroids.  We let $\mathcal{P}_k$ denote the
class of $k$-polymatroids.  If $M_1,M_2,\ldots,M_k$ are matroids on
$E$, then the function $\rho$ on $2^E$ given by
$$\rho(X) = r_{M_1}(X)+ r_{M_2}(X)+\cdots+ r_{M_k}(X),$$ for
$X\subseteq E$, is a $k$-polymatroid on $E$.  We write this as
$\rho = r_{M_1}+ r_{M_2}+\cdots+ r_{M_k}$ for brevity.
We call the multiset
$\{M_1,M_2,\ldots,M_k\}$ a $k$-\emph{decomposition} of $\rho$, and we
say that $\rho$ is \emph{$k$-decomposable}.  We let $\mathcal{D}_k$ be
the class of all $k$-decomposable polymatroids.  Murty and Simon
\cite{murtysimon} raised the question of which polymatroids are
decomposable, and they showed, by example, that not all
$2$-polymatroids are decomposable.

For a polymatroid $\rho$ on $E$ and for $A\subseteq E$, the
\emph{deletion $\rho_{\del A}$} and \emph{contraction $\rho_{/A}$},
both on $E-A$, are defined by $\rho_{\del A}(X) = \rho(X)$ and
$\rho_{/A}(X) = \rho(X\cup A)-\rho(A)$ for all $X\subseteq E-A$. The
\emph{minors} of $\rho$ are the polymatroids of the form
$(\rho_{\del A})_{/B}$ (equivalently, $(\rho_{/B})_{\del A}$) for
disjoint subsets $A$ and $B$ of $E$.  It is easy to check that if
$\rho = r_{M_1}+ r_{M_2}+\cdots+ r_{M_k}$, then
\begin{equation}\label{eq:mindelcon}
  \rho_{\del A} = r_{M_1\del A}+ r_{M_2\del A}+\cdots+ r_{M_k\del A}
  \quad\text{and}\quad \rho_{/A} = r_{M_1/A}+ r_{M_2/A}+\cdots+
  r_{M_k/A},
\end{equation}
so each class $\mathcal{D}_k$ is minor-closed (i.e., all minors of
polymatroids in the class are in the class), as is the union of these
classes, which we denote by $\mathcal{D}$.  Thus, their excluded
minors (i.e., the minor-minimal polymatroids that are not in the
class) are of interest.  (The class $\mathcal{P}_k$ of
$k$-polymatroids is also minor-closed; it has one excluded minor for
each integer $t>k$, namely, a polymatroid consisting of a single
element of rank $t$.)

In our first main result, Theorem \ref{thm:chiconnection2h}, given a
hypergraph $H$ that satisfies certain mild conditions, we construct
polymatroids for which the ordered $k$-decompositions of each of these
polymatroids correspond bijectively to the $k$-colorings of the line
graph of $H$.
One
corollary is that if the line graph is critical (each single-edge
deletion has smaller chromatic number than the graph), then the
polymatroids are excluded minors for some of the classes
$\mathcal{D}_k$.  We also show that certain truncations of some of
these polymatroids are not decomposable.  We give examples to
illustrate the abundance of excluded minors for the classes
$\mathcal{D}$ and $\mathcal{D}_k$ that result.  This casts doubt on
the feasibility of characterizing $\mathcal{D}$ and $\mathcal{D}_k$ by
excluded minors, but it motivates the next definition.

\begin{dfn}\label{def:chromatic}
  The \emph{chromatic polynomial}, $\chi(\rho;x)$, of a polymatroid
  $\rho$ is the function that, when $x$ is a positive integer $k$,
  gives the number of $k$-tuples $(M_1,M_2,\ldots,M_k)$ of matroids
  with $\rho = r_{M_1}+ r_{M_2}+\cdots+ r_{M_k}$.  The \emph{chromatic
    number} $\chi(\rho)$ of $\rho$ is the least positive integer $k$
  such that $\chi(\rho;k)>0$; we set $\chi(\rho)=\infty$ if there is
  no such $k$.
\end{dfn}

By Theorem \ref{thm:chiconnection2h}, for certain hypergraphs
$H$ and any of the polymatroids $\rho$ that we construct from $H$, we
have $\chi(\rho)=\chi(G_H)$ and $\chi(\rho;k)=\chi(G_H;k)$ where $G_H$
is the line graph of $H$.  We give some basic properties of
$\chi(\rho)$ and $\chi(\rho;x)$ in Section \ref{sec:chi}, including
showing that $\chi(\rho;x)$ is a polynomial. 

Our second main result, Theorem \ref{thm:graphbehind2poly}, supports
the observation that $2$-polymatroids have more in common with graphs
than do general polymatroids.  We show that for any $2$-polymatroid
$\rho$, there is a graph $G$ and a rational number $s$ for which
$\chi(\rho;x)=s\cdot \chi(G;x)$. The papers of Lemos
\cite{manoel1,manoel} and Lemos and Mota \cite{LemosMota} provide key
ingredients in the proof.

In Section \ref{sec:duality}, we show how, from a $k$-decomposition of
an $i$-polymatroid $\rho$, where $k\geq i$, we get a $k$-decomposition
of the $i$-dual $\rho^*$ of $\rho$.  We apply this result to excluded
minors and we identify conditions under which we have
$\chi(\rho;x) =\chi(\rho^*;x)$.

In Section \ref{sec:quo}, we explore decompositions of polymatroids
using matroids that form a chain under the quotient
order.
In Theorem \ref{thm:quotkpolyExMin} and
Corollary \ref{cor:quotientexmin}, we give the excluded minors for
minor-closed classes of these polymatroids.

Our notation follows Oxley \cite{oxley}, with two additions: we let
$U_{r,E}$ denote the rank-$r$ uniform matroid on the set $E$ (so its
bases are all $r$-subsets of $E$), and we let $[n]$ denote the set
$\{1,2,\ldots,n\}$.

\section{A construction of polymatroids from
  hypergraphs}\label{sec:genboolh}

A \emph{hypergraph} is an ordered pair $H=(E,\mathcal{E})$ where $E$
is a finite set and $\mathcal{E}$ is a set $\{X_i\,:i\in[n]\}$ of
nonempty subsets of $E$; more broadly, $H$ is a
\emph{multi-hypergraph} if $\mathcal{E}$ is a finite multiset of
subsets of $E$.  The elements of $E$ are the \emph{vertices} of $H$,
and the sets $X_i$ in $\mathcal{E}$ are the \emph{hyperedges} of $H$.
The \emph{line graph} $G_H$ of $H$ has $[n]$ as its vertex set, and
$ij$ is an edge precisely when $i\ne j$ and
$X_i\cap X_j\ne \emptyset$.  Given $H$, define $w:2^E\to \mathbb{Z}$
by letting $w(e)$, for $e\in E$, be the number of sets in
$\mathcal{E}$ that contain $e$, and for $A\subseteq E$, setting
$$w(A) = \sum_{e\in A}w(e)=\sum_{i=1}^n|A\cap X_i|.$$

Some of the results below are limited to hypergraphs that have one or
more of the following special properties:
\begin{itemize}
\item[(H1)] each vertex is in at least two hyperedges, so $w(e)\geq 2$
  for all $e\in E$,
\item[(H2)] if $i,j\in[n]$ with $i\ne j$, then $|X_i\cap X_j|\leq 1$,
\item[(H3)] no three hyperedges have the form
  $\{a,b\},\{a,c\},\{b,c\}$ for $a,b,c\in E$. 
\end{itemize}  

Given a multi-hypergraph $H=(E,\{X_i\,:i\in[n]\})$, for each
$i\in[n]$, fix an integer $t_i$ where $0\leq t_i\leq |X_i|$.
When $H$
is a hypergraph, we sometimes require that
\begin{itemize}
\item[(T)] for each $i\in[n]$, either $t_i=1$ or
  $\min\{w(e)\,:\,e\in X_i\}\leq t_i<|X_i|$.
\end{itemize}
Let the polymatroid $\rho$ be given by
\begin{equation}\label{eq:genbpmh}
  \rho = r_{M_1}+ r_{M_2}+ \cdots+ r_{M_n} \qquad\text{where}\qquad
  M_i=U_{t_i,X_i}\oplus U_{0,E-X_i}. 
\end{equation}
Thus, for $A\subseteq E$,
\begin{equation}\label{eq:genbpmhrk}
  \rho(A) = \sum_{i=1}^n\min\{|A\cap X_i|,\,t_i\}.
\end{equation}
If $t_i\geq 1$ for all $i\in [n]$, then $\rho(e)=w(e)$ for all
$e\in E$.

When $t_i\geq 1$ for all $i\in [n]$ and property (H2) holds, two
useful observations follow.  First, for each hyperedge $X_h$ and set
$A\subseteq X_h$,
\begin{equation}\label{eq:ranksgenh}
  \rho(A) = \min\{w(A),\, w(A)-|A|+t_h\}.
\end{equation}
Second, if $e$ and $f$ are distinct elements of $E$, then
$\rho(\{e,f\})=w(\{e,f\})-\delta$ where $\delta=1$ if there is a
(unique) hyperedge $X_i$ with $e,f\in X_i$ and $t_i=1$, and otherwise
$\delta=0$.

Before turning to our first main result, Theorem
\ref{thm:chiconnection2h}, we make connections with earlier work, give
examples, and develop the idea of incidence sets.

For a multi-hypergraph $H=(E,\mathcal{E})$, let all $t_i$ be $1$ and
set $\psi(e) = \{i\,:\,e\in X_i\}$ for each $e\in E$, so
$\rho(A) = \bigl|\cup_{e\in A}\psi(e)\bigr|$ for all $A\subseteq E$.
These are the \emph{covering hypermatroids} that Helgason
\cite{helgason} introduced, which are now called \emph{Boolean
  polymatroids} or \emph{transversal polymatroids}. Mat\'u\v s
\cite{Matus} found the excluded minors for Boolean polymatroids.

\begin{example}\label{ex:booleanp}
  For a graph $G=(V,E)$ with $V=\{v_i\,:\,i\in[n]\}$ and with no
  isolated vertices, let $X_i$ be the set of edges of $G$ that are
  incident with vertex $v_i$.  If $G$ is simple (i.e., has no loops
  and no parallel edges) and no component is a $3$-cycle, then the
  hypergraph $H=(E,\{X_i\,:\,i\in[n]\})$ satisfies properties
  (H1)--(H3).  The line graph $G_H$ is isomorphic to $G$ and $w(e)=2$
  for all $e\in E$.  Property (T) holds if either $t_i=1$ or
  $0< t_i<|X_i|$ for each $i\in[n]$.  The most-studied case is the
  \emph{Boolean $2$-polymatroid} $\rho_G$ of $G$, which is when
  $t_i=1$ for all $i\in[n]$.  In that case, for $X\subseteq E$, we
  have $\rho_G(X)=|V(X)|$ where $V(X)$ is the set of vertices that are
  incident with at least one edge in $X$.  Thus, $X$ is a matching if
  and only if $\rho_G(X)=2|X|$.  Oxley and Whittle \cite{exminbpm}
  found the excluded minors for Boolean $2$-polymatroids.
\end{example}

\begin{example}\label{ex:affine}
  Let $M$ be a finite affine plane, that is, a rank-$3$ simple matroid
  in which, given any point $x$ and line $L$ with $x\not\in L$, there
  is exactly one line $L'$ with $x\in L'$ and
  $L\cap L'=\emptyset$.
  Let $E$
  be the set of lines of $M$.  For each point $a$ in $E(M)$, let $X_a$
  be the set of all lines of $M$ that contain $a$, and let
  $\mathcal{E}$ be $\{X_a\,:\,a\in E(M)\}$.  The hypergraph
  $H_M= (E,\mathcal{E})$ satisfies properties (H1)--(H3).  For
  distinct points $a$ and $b$ in $E(M)$, the line $\cl_M(\{a,b\})$ is
  in $X_a\cap X_b$, so the line graph of $H_M$ is complete.  The
  \emph{order} of a finite affine plane is the number of points on
  each line.  If $M$ has order $q$, then $w(L)=q$ for all $L\in E$.
  Basic counting gives $|E| = q^2+q$ as well as $|\mathcal{E}|=q^2$
  and $|X_a|=q+1$ for all $a\in E(M)$.  Condition (T) holds if and
  only if $t_a\in\{1,q\}$ for all $a\in E(M)$.

  In place of an affine plane, we can use any rank-$3$ simple matroid
  other than $U_{3,3}$ (if we want property (H3) to hold).  The choice
  of the matroid influences the options for the integers $t_i$.  The
  line graph is always complete.
\end{example}

\begin{example}\label{ex:projpl}
  Let $M$ be a projective plane of order $q$ (i.e., each line has
  $q+1$ points).  Let $E$ be its set of points and $\mathcal{E}$ its
  set of lines.  The line graph of the hypergraph $(E,\mathcal{E})$ is
  the complete graph on $q^2+q+1$ vertices.  Properties (H1)--(H3)
  hold.  Only setting all $t_i=1$ satisfies condition (T) since
  $|X_i|=q+1=w(e)$ for all $X_i\in \mathcal{E}$ and
  $e\in E$.
  
  In place of a projective plane, we can use any rank-$3$ simple
  matroid, provided that (if property (H3) is to hold) no triple of
  trivial lines has the form $\{a,b\}$, $\{a,c\}$, $\{b,c\}$.  The
  choice of $M$ influences the options for the integers $t_i$.  The
  line graph is complete if and only if $M$ is modular, that is, $M$
  is a projective plane or $U_{1,1}\oplus U_{2,n}$ for some
  $n\geq 3$.
\end{example}

\begin{dfn}\label{dfn:incidencesets}
  An \emph{incidence set} of a polymatroid $\rho$ on $E$ is a subset
  $X$ of $E$ with $|X|\geq 2$ for which, for all $Y\subseteq X$ with
  $|Y|\in[3]$,
  \begin{equation}\label{eq:definc}
    \rho(Y) =1+\sum_{e\in Y}\bigl(\rho(e)-1\bigr).
  \end{equation}
\end{dfn}

For example, if $t_h=1$ and property (H2) holds, then the hyperedge
$X_h$ is an incidence set by equation (\ref{eq:ranksgenh}).  In
particular, in the Boolean polymatroid of a simple graph, the set of
edges that are incident with a given vertex of degree at least two is
an incidence set.  If $\rho(e)=0$, then equation (\ref{eq:definc})
fails for any set $\{e,f\}$, so all elements in incidence sets have
positive rank.

\begin{lemma}\label{lem:forceparallel}
  Let $D=\{M_i\,:\,i\in[k]\}$ be a decomposition of a polymatroid
  $\rho$ on $E$.
  \begin{enumerate}
  \item For any incidence set $X$ of $\rho$, there is exactly one
    $i\in[k]$ for which all elements in $X$ are parallel in $M_i$.  No
    two elements of $X$ are parallel in any other matroid in $D$.  If
    equation \emph{(\ref{eq:definc})} also holds for $Y=X$,
    then the only subsets of $X$ that are circuits of
    any other matroid in $D$ are singletons (i.e., loops).
  \item Assume that $X$ and $Y$ are incidence sets and that there are
    elements $a\in X$ and $b\in Y$ with
    $\rho(\{a,b\})=\rho(a)+\rho(b)$.  If $X\cap Y\ne \emptyset$, then
    $|X\cap Y|=1$, and the matroid in $D$ in which the elements of $Y$
    are parallel differs from that for $X$.
\end{enumerate}
\end{lemma}

\begin{proof}
  For any $e,f \in X$, there are exactly $\rho(e)$ matroids $M_i$ in
  $D$ with $r_{M_i}(e)=1$, and exactly $\rho(f)$ matroids $M_j$ in $D$
  with $r_{M_j}(f)=1$, yet $\rho(\{e,f\}) = \rho(e)+\rho(f)-1$, so $e$
  and $f$ must be parallel in exactly one matroid in $D$.  If
  $|X|\geq 3$, then for any third element $g\in X$, the same
  conclusion applies to $e$ and $g$, and to $f$ and $g$.  If these
  pairs were parallel in three different matroids, then
  $\rho(\{e,f,g\})\leq \rho(e)+\rho(f)+\rho(g) -3$, contrary to
  equation (\ref{eq:definc}).  Thus, two of the pairs are parallel in
  the same matroid, and so all three are parallel in that matroid.
  This applies to any triple in $X$, so the first two assertions in
  part (1) follow.  Similar ideas yield the last assertion in part
  (1).

  For part (2), since $\rho(\{a,b\})=\rho(a)+\rho(b)$, these elements
  are not parallel in any matroid in $D$.  Let $e\in X\cap Y$.  The
  parallel class that $X$ is in contains $a$ and $e$, and the one that
  $Y$ is in contains $b$ and $e$, so, with $e$ in common, these
  parallel classes must be in different matroids.  If
  $f\in (X\cap Y)-e$, then $f$ would also be in both of those parallel
  classes, but this contradicts part (1) since $e$ and $f$ are
  parallel in just one matroid in $D$.
\end{proof}

\begin{example}\label{ex:Vamos1}
  We define counterparts, for $k$-polymatroids, of the V\'amos
  matroid. 
    Let $E$ be $\{a,b,c,d\}$.  For integers $k\geq 2$ and
  $s$ with $3k-2\leq s\leq 4k-4$, let $\rho_{k,s}$ be given by
  $$\rho_{k,s}(X) =
  \begin{cases}
    k\,|X|, & \text{if } |X|\leq 1  \text{ or } X=\{a,d\},\\
    2k-1, & \text{if } |X|=2 \text{ and } X\ne\{a,d\},\\
    3k-2, & \text{if } |X|=3,\\
    s, & \text{if } X=E.
  \end{cases}$$ Figure \ref{fig:vamos} shows $\rho_{2,4}$.  It is easy
  to check that, if $\rho$ is a $j$-polymatroid on $E$ and $k\geq j$,
  then $\rho^k$, defined by $\rho^k(X)=\rho(X)+(k-j)|X|$ for
  $X\subseteq E$, is a $k$-polymatroid. Note that
  $\rho_{k,4k-4}=\rho_{2,4}^k$.  This operation appears again in
  Section \ref{sec:duality}.  Also, $\rho_{k,s}$ is a truncation of
  $\rho_{k,4k-4}$.  Truncations are discussed at greater length at the
  end of this section.  The sets $X = \{a,b,c\}$ and $Y = \{b,c,d\}$
  are incidents sets of $\rho_{k,s}$, yet $|X\cap Y|=2$ and
  $\rho_{k,s}(\{a,d\}) = \rho_{k,s}(a) +\rho_{k,s}(d)$, so
  $\rho_{k,s}$ is not decomposable by part (2) of Lemma
  \ref{lem:forceparallel}.  It is easy to check that $\rho_{k,s}$ is
  an excluded minor for $\mathcal{D}$ and for $\mathcal{D}_t$ for all
  $t\geq k$.  Variations on this example yield more excluded minors
  for decomposability.
\end{example}

\begin{figure}
  \centering
\begin{tikzpicture}[scale=1]
\draw[thick,black!20](0,1)--(1,0.1)--(2.5,0.1)--(3,1);
\draw[thick,black!20](0.7,1.4)--(1,0.1)--(2.5,0.1)--(3.7,1.4);
\draw[very thick](1,0.1)--(2.5,0.1);
\draw[line width = 3.5pt,white](0.6,1)--(1.2,1);
\draw[thick,black!20](0,1)--(3,1)--(3.7,1.4)--(0.7,1.4)--(0,1);
\draw[very thick](0,1)--(3,1);
\draw[very thick](3.7,1.4)--(0.7,1.4);
\draw[line width = 3.5pt,white](2.85,1.3)--(2.85,1.5);
\draw[thick,black!20](0,1)--(1,2.3)--(2.5,2.3)--(3,1);
\draw[thick,black!20](0.7,1.4)--(1,2.3)--(2.5,2.3)--(3.7,1.4);
\draw[very thick](1,2.3)--(2.5,2.3);

\node at (1.8,2.45) {\footnotesize $a$};
\node at (1.8,0.8) {\footnotesize $b$};
\node at (1.8,1.55) {\footnotesize $c$};
\node at (1.8,-0.1) {\footnotesize $d$};
\end{tikzpicture}
\caption{A $2$-polymatroid counterpart of the V\'amos matroid.  }
  \label{fig:vamos}
\end{figure}
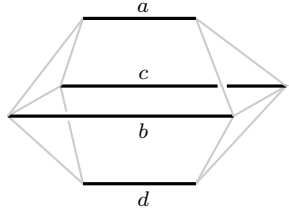

We turn to our first main result.  For a polymatroid $\rho$ on $E$ and
a positive integer $k$, let
$$\Delta_\rho^k
=\{(N_1,N_2,\ldots,N_k)\,:\, N_i \text{ is a matroid on } E \text{ and
} \rho = r_{N_1}+ r_{N_2}+ \cdots+ r_{N_k}\}.$$ Thus,
$\chi(\rho;k) =|\Delta_\rho^k|$.

\begin{thm}\label{thm:chiconnection2h}
  Assume that the hypergraph $H=(E,\{X_i\,:\,i\in[n]\})$ and integers
  $t_i$ satisfy properties \emph{(H2)}, \emph{(H3)}, and
  \emph{(T)}.
  Let $\rho$ be given by equation \emph{(\ref{eq:genbpmh})}.  For each
  positive integer $k$, there is a bijection from the set of
  $k$-colorings $c:[n] \to[k]$ of the line graph $G_H$ onto
  $\Delta_\rho^k$.  Thus, the least $k$ with $\rho\in\mathcal{D}_k$ is
  $k=\chi(G_H)$, and $ |\Delta_\rho^k|= \chi(G_H;k)$ for all positive
  integers $k$.
\end{thm}

This result follows from the next two, which strengthen different
parts of the statement.  The next lemma, which is an immediate
consequence of the definitions, assumes none of properties (H1)--(H3)
and (T).

\begin{lemma}\label{lem:colortosum}
  Let $H=(E,\{X_i\,:\,i\in[n]\})$ be a multi-hypergraph.  Let $\rho$
  be given by equation \emph{(\ref{eq:genbpmh})} where
  $0\leq t_i\leq |X_i|$ for each $i\in[n]$.  Define a map $\phi$ on
  the set of $k$-colorings $c:[n] \to[k]$ of the line graph $G_H$ by
  $\phi(c) = (N_1,N_2,\ldots,N_k)$ where
  \begin{equation}\label{eq:dirsumdecompchih}
    N_i=\Biggl(\bigoplus_{h\,:\,c(h)=i}U_{t_h,X_h}\Biggr)
    \oplus U_{0,Y}
    \quad \text{ where } \quad Y=E-\bigcup_{h\,:\,c(h)=i}X_h.
  \end{equation}
  Then $\phi(c)\in \Delta_\rho^k$.  Thus, $\chi(\rho)\leq\chi(G_H)$.
  If $H$ is a hypergraph and, for all $i\in[n]$, either $t_i=1$ or
  $0<t_i<|X_i|$, then $\phi$ is injective.
\end{lemma}

The heart of the proof of Theorem \ref{thm:chiconnection2h} is showing
that, when properties (H2), (H3), and (T) hold, the image of the map
$\phi$ in Lemma \ref{lem:colortosum} is $\Delta_\rho^k$.  That is part
of the next result, which, for any polymatroid $\sigma$ on $E$ for
which $\sigma(A)=\rho(A)$ for certain sets $A$, produces colorings of
$G_H$ from decompositions of $\sigma$.

For a $k$-coloring $c$ of a graph, let $c_1$ and $c_2^+$ be the number
of $i\in [k]$ for which the size of the preimage $c^{-1}(i)$ is,
respectively, $1$ and at least $2$.

\begin{thm}\label{thm:prepfortruncate}
  Assume that the hypergraph $H=(E,\{X_i\,:\,i\in[n]\})$ and integers
  $t_i$ satisfy properties \emph{(H2)}, \emph{(H3)}, and \emph{(T)}.
  Let $\rho$ be given by equation \emph{(\ref{eq:genbpmh})}.  Let
  $\sigma$ be any polymatroid on $E$ for which $\sigma(A)=\rho(A)$ if
  $A\subseteq E$ and
  \begin{itemize}
  \item $|A|\leq 3$, or
  \item $|A|=4$ and there are three hyperedges $X_h$, $X_i$, $X_j$
    with $t_h=t_i=t_j=1$ so that some element of $A$ is in exactly one
    of these hyperedges, and each of the other three elements of $A$
    is in exactly two of these hyperedges, or
  \item $A\subseteq X_i$ and $|A|\leq t_i+1$ for some $i\in [n]$, or
  \item $A$ is a subset of the symmetric difference $X_i\btu X_j$, for
    some $i,j\in[n]$ with $X_i\cap X_j\ne\emptyset$, where
    $|A\cap X_i|\leq t_i$ and $|A\cap X_j|\leq t_j$.
  \end{itemize}
  From an element in $\Delta_\sigma^k$, we can construct a
  $k$-coloring $c$ of the line graph $G_H$; moreover,
  $\sigma(E)\geq c_1+2c_2^+$.  If $\sigma=\rho$, then the image of the
  injection $\phi$ in Lemma \ref{lem:colortosum} is $\Delta_\rho^k$,
  so $\phi$ is the bijection asserted in Theorem
  \ref{thm:chiconnection2h}.  If $\sigma(E)<\chi(G_H)$, then $\sigma$
  is indecomposable.
\end{thm}

\begin{proof}
  Fix $(N_1,N_2,\ldots,N_k)\in \Delta_\sigma^k$.  Since
  $\sigma(e)=\rho(e)=w(e)$ for each $e\in E$, we have $r_{N_i}(e)=1$
  for exactly $w(e)$ integers $i\in[k]$.
  Consider a
  hyperedge $X_h$ with $|X_h|>1$.  If $t_h=1$, then $X_h$ is an
  incidence set of $\sigma$, so, by Lemma \ref{lem:forceparallel},
  there is exactly one $i\in[k]$ for which $N_i|X_h=U_{t_h,X_h}$.  We
  claim that the same holds when $t_h>1$.  To see this, fix $h$ with
  $t_h>1$, so $t_h<|X_h|$.  From equation (\ref{eq:ranksgenh}) and the
  hypothesis, for $A\subseteq X_h$, we have
  \begin{itemize}
  \item[(1)] if $|A|\leq t_h$, then $\sigma(A)=\rho(A)=w(A)$, so for
    all $j\in[k]$, the restriction $N_j|A$ has only loops and coloops,
    and
  \item[(2)] if $|A|=t_h+1$, then $\sigma(A)=\rho(A)=w(A)-1$, so $A$
    is a circuit of exactly one matroid $N_i$ with $i\in[k]$.
  \end{itemize}
  The claim follows if $|X_h|=t_h+1$.  When $|X_h|\geq t_h+2$, fix
  $A\subseteq X_h$ with $|A|=t_h+2$ and containing an element $f$ with
  $w(f) \leq t_h$.  The set $A$ has $t_h+1$ subsets of size $t_h+1$
  that contain $f$.  Each such subset of $A$ is a circuit of some
  matroid $N_i$.  Now $w(f)\leq t_h$, so $r_{N_i}(f)=1$ for at most
  $t_h$ matroids $N_i$.  Thus, there are two elements $a$ and $a'$ in
  $A-f$ and an integer $i\in[k]$ so that both $A-a$ and $A-a'$ are
  circuits of $N_i$.  Applying circuit elimination to these two
  circuits of $N_i$ and using deduction (1) above shows that all
  $(t_h+1)$-element subsets of $A$ are circuits of $N_i$, that is,
  $N_i|A=U_{t_h,A}$.  If $|X_h|\geq t_h+3$ and $b \in X_h-A$, then fix
  $x\in A-f$ and let $A'=(A\cup b)-x$.  By the same argument,
  $N_j|A'=U_{t_h,A'}$ for some $j\in [n]$; since $|A\cap A'|=t_h+1$,
  we have $j=i$ by deduction (2).  Sequences of such exchanges give
  $N_i|X_h=U_{t_h,X_h}$.

  Define $c:[n]\to[k]$ by, for each $h\in [n]$, setting $c(h)=i$ where
  \begin{itemize}
  \item if $|X_h|>1$, then $i$ is the unique integer in $[k]$ for
    which $N_i|X_h=U_{t_h,X_h}$, and
  \item if $X_h=\{e\}$, then $i$ is the unique element in the
    difference
    $$\{j\,:\,r_{N_j}(e)=1\}-\{c(s)\,:\,e\in X_s \text{ and }
    |X_s|>1\}.$$
  \end{itemize}
  We next show that $c$ is a coloring of $G_H$.  The argument implies
  that if $X_h=\{e\}$, then
  $|\{c(s)\,:\,e\in X_s \text{ and } |X_s|>1\}|=w(e)-1$, so the
  difference above contains exactly one element, and so $c$ is
  well-defined.

  Let vertices $p$ and $q$ be adjacent in $G_H$, so
  $X_p\cap X_q\ne \emptyset$.  Assumption (H2) gives
  $X_p\cap X_q=\{e\}$ for some $e\in E$.  We show by contradiction
  that, as required of a coloring, $c(p)\ne c(q)$.  We may assume that
  $X_p\ne\{e\}$ and $X_q\ne\{e\}$, for otherwise the conclusion
  follows by the definition of $c$.  Assume that $ c(p)= c(q)=i$.  Let
  $C_p$ and $C_q$ be circuits of the uniform matroids $N_i|X_p$ and
  $N_i|X_q$, respectively, with $e\in C_p\cap C_q$.  There is a
  circuit $C$ of $N_i$ with $C\subseteq (C_p\cup C_q)-e$ by circuit
  elimination.  Let $\alpha$ be the number of pairs of elements in
  $C$, one in $X_p$ and one in $X_q$, that are also in some other
  hyperedge, say $X_j$, that has $t_j=1$.  By property (H2), each such
  pair is in just one hyperedge and no two such pairs are in the same
  hyperedge.  Since $|C\cap X_p|\leq t_p$ and $|C\cap X_q|\leq t_q$,
  equation (\ref{eq:genbpmhrk}) and the hypotheses give
  $\sigma(C)=\rho(C)=w(C)-\alpha$.  First assume that $|C|>2$.  Thus,
  either $t_p>1$ or $t_q>1$.  Also, no two elements of $C$ are
  parallel in $N_i$, so the pairs that $\alpha$ counts are parallel in
  matroids $N_j$ with $j\in[k]-i$.  Furthermore, no two pairs that
  $\alpha$ counts and that share an element are parallel in the same
  matroid, for if $\{a,b\}$ and $\{a,c\}$ are two such pairs with, say
  $b,c\in X_q$, and both are circuits of $N_m$, then $b$ and $c$ are
  parallel in $N_m$ and so, by the observation after equation
  (\ref{eq:ranksgenh}), are in a hyperedge $X_h$ with $t_h=1$;
  however, since $b,c\in X_q$, we have $h=q$, but
  $2\leq |C\cap X_q|\leq t_q$, so $t_q\ne 1$.  We now get
  $$r_{N_1}(C)+ r_{N_2}(C)+ \cdots+ r_{N_k}(C)\leq
  w(C)-1-\alpha< \sigma(C).$$ This contradiction shows that
  $C=\{e_p,e_q\}$ for some $e_p\in X_p$ and $e_q\in X_q$.
  
  By the observation after equation (\ref{eq:ranksgenh}), since $C$ is
  a circuit of $N_i$, there is a hyperedge $X_s$ with $C\subseteq X_s$
  and $t_s=1$.  Also, $s\not\in\{p,q\}$.  Since $t_s=1$, the elements
  in $X_s$ are parallel in exactly one matroid, which $C$ shows is
  $N_i$.

  Since $C$ is a circuit of $N_i$, if either $\{e,e_p\}$ or
  $\{e,e_q\}$ is a circuit of $N_i$, then both are.  Thus, $t_p=1$ if
  and only if $t_q=1$.  The sublemma below implies that $t_p>1$ and
  $t_q>1$.
  
  \begin{sublemma}\label{lem:t_i=1coloring}   
    Assume that $\{a,b,c\}\subseteq [n]$ with $t_a=t_b=t_c=1$.  If
    $F=\{f_a, f_b, f_c\}$ is a set of three vertices with
    $\{f_b,f_c\}\subseteq X_a$, $\{f_a,f_c\}\subseteq X_b$, and
    $\{f_a,f_b\}\subseteq X_c$, then there is no $i\in[k]$ with
    $N_i|F=U_{1,F}$.
  \end{sublemma}

  \begin{proof}
    Assume instead that $N_i|F= U_{1,F}$.  Thus, $N_j|X_a=U_{1,X_a}$
    if and only if $j=i$ by Lemma \ref{lem:forceparallel}, and
    likewise for $X_b$ and $X_c$.  By property (H3) and symmetry, we
    may assume that $ \{f_b,f_c\}\subsetneq X_a$.  Let $F'=F\cup f$
    where $f\in X_a-\{f_b,f_c\}$.  Now $\sigma(X)=\rho(X)$ if
    $X\subseteq F'$.  We will derive the contradiction that $\sigma$
    and $r_{N_1}+ r_{N_2}+ \cdots+ r_{N_k}$ differ on $F'$.

    We first consider the $3$-element subsets of $F'$.  By the
    hypotheses, equation (\ref{eq:genbpmhrk}) applies to all subsets
    of $F'$, so, since $t_a=1$,
    $$\sigma(\{f_b,f_c,f\})=1+\bigl(w(f_b)-1\bigr)+\bigl(w(f_c)-1\bigr)
    +\bigl(w(f)-1\bigr) =w(\{f_b,f_c,f\})-2.$$ Thus, since
    $r_{N_i}(\{f_b,f_c,f\})=1$, the restriction of each other matroid
    $N_j$ to $\{f_b,f_c,f\}$ has only loops and coloops.  Equation
    (\ref{eq:genbpmhrk}) also gives
    $$\sigma(F)=3+\bigl(w(f_a)-2\bigr)+\bigl(w(f_b)-2\bigr)
    +\bigl(w(f_c)-2\bigr) =w(F)-3,$$ since each of $X_a$, $X_b$, and
    $X_c$ contains exactly two of $f_a,f_b,f_c$.  Thus, since
    $r_{N_i}(F)=1$, the set $F$ is a circuit of exactly one other
    matroid $N_j$.  Since $f_a$ and $f$ are parallel in $N_i$, some
    hyperedge $X_d$ with $t_d=1$ must contain $f_a$ and $f$.  Thus,
    the same argument also shows that $\{f_a,f_c,f\}$ is a circuit of
    exactly one of the matroids other than $N_i$, as is
    $\{f_a,f_b,f\}$.  No two of $F$, $\{f_a,f_c,f\}$, and
    $\{f_a,f_b,f\}$ can be circuits of the same matroid $N_j$ since
    that would give the contradiction that either $\{f_b,f_c,f\}$ or
    one of its two-element subsets is a circuit of $N_j$.  Since
    $t_a=t_b=t_c=t_d=1$, equation (\ref{eq:genbpmhrk}) gives
    $$\sigma(F')=4+\bigl(w(f_a)-3\bigr)+\bigl(w(f_b)-2\bigr)
    +\bigl(w(f_c)-2\bigr) +\bigl(w(f)-2\bigr)=w(F')-5.$$ Thus, since
    $F$, $\{f_a,f_c,f\}$, and $\{f_a,f_b,f\}$ are circuits of three
    different matroids and $N_i|F'= U_{1,F'}$, we get the claimed
    contradiction,
    \begin{equation*}
      r_{N_1}(F')+ r_{N_2}(F')+ \cdots+ r_{N_k}(F') \leq w(F')-6 <
      \sigma(F').\qedhere 
    \end{equation*}
  \end{proof}

  Thus, $t_p>1$.  Fix $A \subseteq X_p$ with $|A|=t_p$ and
  $e_p\not\in A$.  Thus, $A\cup e_p$ is a circuit of $N_i$, as is
  $A\cup e_q$ since $e_p$ and $e_q$ are parallel in $N_i$.  Since
  $t_p>1$, the only element of $X_p$ that is parallel to $e_q$ in
  $N_i$ is $e_p$. Thus,
  $$r_{N_1}(A\cup e_q)+ r_{N_2}(A\cup e_q)+ \cdots+ r_{N_k}(A\cup
  e_q)\leq w(A)+w(e_q)-1-\delta$$ where $\delta$ is the number of
  elements $a\in A$ for which $a$ and $e_q$ are in another hyperedge,
  say $X_m$, with $t_m=1$ (so $a$ and $e_q$ are parallel in $N_j$ for
  some $j\ne i$).  This is a contradiction since
  $\sigma(A\cup e_q)=\rho(A\cup e_q) = w(A)+w(e_q)-\delta$ by the
  hypothesis and equation (\ref{eq:genbpmhrk}).

  Thus, $c$ is a coloring of $G_H$.
  If
  $|c^{-1}(i)|=1$, then $r(N_i)\geq 1$.  Assume that
  $|c^{-1}(i)|\geq 2$; say $c(p)=c(q)=i$ with $p\ne q$.  If
  $r(N_i)=1$, then any $x\in X_p$ and $y\in X_q$ are parallel in
  $N_i$, so they are in some hyperedge $X_j$ with $t_j=1$.  The
  elements in $X_j$ are parallel in exactly one matroid, which $x$ and
  $y$ show is $N_i$, so $c(j)=i$, but this is impossible since $c$ is
  a coloring of $G_H$ and $X_p\cap X_j\ne\emptyset$.  Thus,
  $r(N_i)\geq 2$.  Therefore
  $\sigma(E)\geq c_1+2c_2^+\geq c_1+c_2^+\geq \chi(G_H)$.  Finally,
  $\rho(E) = t_1+t_2+\cdots+t_n$ by equation (\ref{eq:genbpmhrk}), so
  when $\sigma=\rho$, each $N_i$ must be the direct sum of its maximal
  restrictions to uniform matroids, so equation
  (\ref{eq:dirsumdecompchih}) holds and
  $\phi(c)=(N_1,N_2,\ldots,N_k)$.
\end{proof}

\begin{cor}
  The hypergraph $H$ can be obtained from the polymatroid
  $\rho$.
\end{cor}

\begin{proof}
  Given a decomposition $\{N_i\,:\,i\in[k]\}$ of $\rho$, for each
  $i\in [k]$, delete the loops of $N_i$; the ground sets of the
  connected components of these matroids are the hyperedges.
\end{proof}

Recall that a $(k+1)$-critical graph has chromatic number $k+1$ but
each deletion of one edge has chromatic number $k$.

\begin{cor}\label{cor:criticalgraph2h}
  Assume, in addition to the hypotheses of Theorem
  \ref{thm:chiconnection2h}, that property \emph{(H1)} holds.  If the
  line graph $G_H$ is $(k+1)$-critical, then $\rho$ is an excluded
  minor for $\mathcal{D}_k$.
\end{cor}

\begin{proof}
  Since $\chi(G_H)=k+1$, Theorem \ref{thm:chiconnection2h} gives
  $\rho\not\in \mathcal{D}_k$.  For $e\in E$, let $G_e$ be the
  subgraph of $G_H$ obtained by deleting all edges $ij$ with
  $X_i\cap X_j=\{e\}$.  By properties (H1) and (H2), at least one edge
  of $G_H$ is deleted to get $G_e$.  With equation
  (\ref{eq:mindelcon}) and Lemma \ref{lem:colortosum}, we get
  $\chi(\rho_{\del e})\leq\chi(G_e)$ and
  $\chi(\rho_{/e})\leq\chi(G_e)$, so $\rho_{\del e}\in \mathcal{D}_k$
  and $\rho_{/e}\in \mathcal{D}_k$.
\end{proof}

\begin{ex:bpm2}
  By Corollary \ref{cor:criticalgraph2h}, if $G$ is a $(k+1)$-critical
  graph other than the $3$-cycle, then each $2$-polymatroid that we
  obtain from it is an excluded minor for $\mathcal{D}_k$.
  Thus, finding all $2$-polymatroids that are
  excluded minors for $\mathcal{D}_k$
  is at least as hard as finding
  all $(k+1)$-critical graphs.  The $3$-cycle $C_3$ shows the need for
  property (H3): while $\chi(C_3)=3$, the set $\{U_{1,E},U_{2,E}\}$ is
  a decomposition of $\rho_{C_3}$, so $\chi(\rho_{C_3})=2$.
\end{ex:bpm2}

The next two examples show, in contrast to the case of graphs, that
when $\chi(\rho)$ is finite, the differences
$\chi(\rho)-\chi(\rho_{\del e})$ and $\chi(\rho)-\chi(\rho_{/e})$ can
be arbitrarily large.  A famous open problem asks: are there finite
affine or projective planes of orders that are not powers of primes?
Thus, these examples also confirm the difficulty of characterizing
$\mathcal{D}_k$ by excluded minors.

\begin{ex:aff2} 
  In the example from an affine plane of order $q$, the line graph of
  the hypergraph $H=(E,\mathcal{E})$ is complete, so $\chi(\rho)=q^2$.
  For an element (i.e., line) $L\in E$, the hypergraph for
  $\rho_{\del L}$ and $\rho_{/L}$ is
  $H_L=(E-\{L\},\{X-\{L\}\,:\,X\in \mathcal{E}\})$.  The $q$
  hyperedges of $H$ that contain $L$ give $q$ disjoint hyperedges in
  $H_L$, so $\rho_{\del L}$ and $\rho_{/L}$ are in
  $\mathcal{D}_{q^2-q+1}$ by Lemma \ref{lem:colortosum}.  Thus, $\rho$
  is an excluded minor for $\mathcal{D}_k$ if $q^2-q< k< q^2$.  The
  extension of this example to rank-$3$ simple matroids yields more
  excluded minors for some classes $\mathcal{D}_k$.
\end{ex:aff2}

\begin{ex:ppl2} 
  In the example from a projective plane of order $q$, the line graph
  is complete, so $\chi(\rho)=q^2+q+1$.  The argument in the previous
  example, now using the $q+1$ hyperedges that contain a fixed element
  (i.e., point) $e$, shows that $\rho_{\del e}$ and $\rho_{/e}$ are in
  $\mathcal{D}_{q^2+1}$.  Thus, $\rho$ is an excluded minor for
  $\mathcal{D}_k$ if $q^2< k< q^2+q+1$.  The extension of this example
  to some rank-$3$ simple matroids yields more excluded minors for
  some classes $\mathcal{D}_k$ if the line graph (which need not be
  complete) is critical.
\end{ex:ppl2}

We end this section by exploring some of the polymatroids $\sigma$
that Theorem \ref{thm:prepfortruncate} treats.  For simplicity we
focus on truncations.  For a polymatroid $\rho$ on $E$ and nonnegative
integer $s$ with $s\leq \rho(E)$, the \emph{truncation of $\rho$ to
  rank $s$} is the polymatroid $T (\rho ,s)$ on $E$ that is given by
$T(\rho ,s)(X)= \min\{\rho(X),s\}$ for $X\subseteq E$.  The next
corollary of Theorem \ref{thm:prepfortruncate} identifies many
truncations that are indecomposable.

\begin{cor}\label{cor:colorconnect2gen}
  Let the hypergraph $H$ satisfy properties \emph{(H2)} and
  \emph{(H3)}.  Let $\rho$ be given by equation
  \emph{(\ref{eq:genbpmh})} where each $t_i$ is $1$.  Fix $s$ with
  $\max\{\rho(A)\,:\,A\subseteq E,\, |A|\leq 4 \} \leq s<\rho(E)$.  If
  the truncation $T(\rho ,s)$ is $k$-decomposable, then there is a
  $k$-coloring $c$ of the line graph $G_H$ and $s\geq c_1+2c_2^+$.  If
  $s<\chi(G_H)$, then $T(\rho ,s)$ is indecomposable.
\end{cor}

We now give examples and results that show that special features of
some polymatroids allow one to deduce indecomposability for a wider
range of truncations.

In the special case of the Boolean polymatroid $\rho_G$ of a simple
graph $G=(V,E)$, no component of which is a $3$-cycle, routine graph
arguments show that, when $4\leq s\leq |V|$, one can determine the
incidence sets of $\rho_G$ and deduce that they are parallel classes
of the matroids in any decomposition of the truncation $T (\rho ,s)$.
Thus, if $4\leq s\leq |V|$, then the conclusions in Corollary
\ref{cor:colorconnect2gen} hold for $T (\rho ,s)$, using $G$ in place
of $G_H$.

\begin{example}\label{ex:furthercycle}
  Let $\rho$ be the truncation to rank $4$ of the Boolean polymatroid
  of the cycle $C_{2t+1}$ on $2t+1$ edges, where $t\geq 2$.  We claim
  that $\rho$ is an excluded minor for $\mathcal{D}_k$ for all
  $k\geq 2$.  For any coloring $c$ of $C_{2t+1}$ with $t\geq 2$, we
  have (i) $c_1+c_2^+\geq 5$, or (ii) $c_1+c_2^+=4$ and $c_2^+\geq 1$,
  or (iii) $c_1+c_2^+=3$ and $c_2^+\geq 2$.  Thus, $c_1+2c_2^+\geq 5$,
  so $\rho\not\in\mathcal{D}_k$.  A single-element contraction of
  $\rho$ has rank two and has two rank-one elements; it is easy to
  check that it is in $\mathcal{D}_2$.  A single-element deletion is
  the truncation, to rank $4$, of the Boolean polymatroid of a path,
  say with vertices $v_1,v_2,\ldots,v_{2t+1}$, in order. A
  decomposition of the deletion has two rank-$2$ matroids, one with
  parallel classes (in the notation of Example \ref{ex:booleanp})
  $X_i$ with $i$ odd, and the other with parallel classes $X_i$ with
  $i$ even.
\end{example}

The next example treats a truncation of the Boolean polymatroid of a
tree with $n$ edges and $n-1$ leaves.  Note that property (H2) fails.

\begin{example}\label{ex:infiniteclassforD2}
  For a set $E$ with $|E|=n\geq 4$, let $\{A,B,C\}$ be a partition of
  $E$ with $|C|=1$.  Let the decomposition
  $\{ U_{1,A\cup C}\oplus U_{0,B},\, U_{1,B\cup C}\oplus U_{0,A},\,
  U_{n-2,A\cup B}\oplus U_{0,C}\}$ define $\rho$.  We claim that
  $\rho$ is an excluded minor for $\mathcal{D}_2$.  With equation
  (\ref{eq:mindelcon}), it follows that $\rho_{\del e}$ and
  $\rho_{/e}$ are in $\mathcal{D}_2$ for all $e\in E$.  To see that
  $\rho\notin\mathcal{D}_2$, assume instead that
  $\rho = r_{M_1}+ r_{M_2}$.  Both $A\cup C$ and $B\cup C$ are
  incidence sets of $\rho$, and $\rho(\{a,b\})=\rho(a)+\rho(b)$ for
  all $a\in A$ and $b\in B$.  Thus, Lemma \ref{lem:forceparallel}
  implies that $A\cup C$ is a parallel class of one of the matroids,
  say $M_1$, and $B\cup C$ is a parallel class of $M_2$.  Since
  $\rho(A\cup C)=|A|+2$, we get $r_{M_2}(A\cup C)=|A|+1$; likewise,
  $r_{M_1}(B\cup C)=|B|+1$.  This gives the contradiction
  $$r(M_1)+r(M_2) \geq |A|+1+|B|+1>n=\rho(E).$$
\end{example}

Among the many hypergraphs to which the next result applies are those
that we obtain from affine planes as in Example
\ref{ex:affine}. 

\begin{thm}
  Assume that the hypergraph $H=(E,\{X_i\,:\,i\in[n]\})$ satisfies
  property \emph{(H2)}, that $G_H$ is complete, and that whenever $p$
  and $q$ are distinct elements of $[n]$, there are elements
  $a\in X_p$ and $b\in X_q$ for which no $i\in[n]$ has
  $\{a,b\}\subseteq X_i$.  Let $\rho$ be given by equation
  \emph{(\ref{eq:genbpmh})} where each $t_i$ is $1$.  Let $\sigma$ be
  any polymatroid on $E$ for which $\sigma(A)=\rho(A)$ whenever either
  \emph{(i)} $|A|\leq 2$ or \emph{(ii)} $|A|=3$ and $A\subseteq X_i$
  for some $i\in [n]$.  If $\sigma(E)<n$, then $\sigma$ is
  indecomposable.  In particular, the truncation $T (\rho ,s)$ is
  indecomposable if
  $$\max\{\rho(A)\,:\,|A|\leq 2 \text{ or } A \text{ is a
  }3\text{-element subset of a hyperedge}\}\leq s<n.$$
\end{thm}

\begin{proof}
  Assume that the multiset $D$ is a decomposition of $\sigma$.  Fix
  $p,q\in[n]$ with $p\ne q$.  The hypotheses imply that each of $X_p$
  and $X_q$ has at least two elements and is an incidence set of
  $\sigma$, so, by Lemma \ref{lem:forceparallel}, the elements of
  $X_p$ are parallel in exactly one matroid $N_i$ in $D$, and those in
  $X_q$ are parallel in exactly one matroid $N_j$ in $D$.  For
  elements $a\in X_p$ and $b\in X_q$ for which no hyperedge contains
  $\{a,b\}$, we have $\sigma(\{a,b\}) = \sigma(a)+ \sigma(b)$.  Since
  $X_p\cap X_q\ne \emptyset$, part (2) of Lemma
  \ref{lem:forceparallel} gives $i\ne j$.  Therefore
  $\sigma(E)\geq n$.  However, $\sigma(E) < n$.  Thus, $\sigma$ is
  indecomposable.
\end{proof}

We now return to projective planes.

\begin{ex:ppl2}
  Let $M$ be a projective plane of order $q$ on the set $E$ and let
  $H$ be the hypergraph $(E,\mathcal{E})$ in Example \ref{ex:projpl},
  so $|E|=|\mathcal{E}|=q^2+q+1$.  Let $\rho$ be given by equation
  (\ref{eq:genbpmh}) where each $t_i$ is $1$.  Let $\sigma$ be any
  polymatroid on $E$ for which $\sigma(A)=\rho(A)$ for all
  $A\subseteq E$ with $|A|\leq 3$.  We claim that if
  $\sigma(E)\leq q^2+q$, then $\sigma$ is indecomposable.  When $q=2$,
  each line $L$ has $|L|=3$ and $\rho(L)=7$, so, in order to have
  $\sigma(E)\leq q^2+q$, we assume that $q>2$.  Let the multiset $D$
  of matroids be a decomposition of $\sigma$.  For a subset $X$ of a
  line $L$ of $M$, we have $\rho(X) = 1+q|X|$.  It follows that $L$ is
  an incidence set of $\sigma$.  By Lemma \ref{lem:forceparallel}, the
  points in $L$ are parallel in exactly one matroid in $D$.  Any two
  lines of $M$ intersect in a point, so from $\sigma(E)\leq q^2+q$, it
  follows that there are distinct lines $L$ and $L'$ of $M$ for which
  the points in $L\cup L'$ are parallel in the same matroid $N_i$.
  For any point $a\in E-(L\cup L')$, some line $L''$ that contains $a$
  intersects $L$ and $L'$ in two different points; since the points in
  $L''$ are parallel in exactly one matroid in $D$, and two points in
  $L''$ are in $L\cup L'$ and so are parallel in $N_i$, it follows
  that $a$ is in that parallel class of $N_i$.  Thus, $N_i=U_{1,E}$.
  For $A\subseteq E$ with $|A|=3$, if $r_M(A)=2$, then
  $\sigma(A) = 1+3q = w(A)-2$, otherwise $\sigma(A) = 3q = w(A)-3$.
  Thus, $A$ is a circuit in some (necessarily unique) matroid in $D$
  if and only if $r_M(A)=3$.  Let $L=\{a_1,a_2,\ldots,a_{q+1}\}$ be a
  line of $M$.  For any $x\in E-L$, the circuits
  $\{a_1,a_2,x\},\{a_1,a_3,x\},\ldots,\{a_1,a_{q+1},x\}$ are contained
  in lines $L_2,L_3,\ldots,L_{q+1}$ of matroids
  $N_{j_2},N_{j_3},\ldots,N_{j_{q+1}}$ in $D$.  No three points in $L$
  are in such a line, so $j_2,j_3,\ldots,j_{q+1}$ are distinct.  Since
  $w(a_1)=q+1$, these $q$ matroids along with $N_i$ are the only ones
  in which $a_1$ has positive rank.  This applies for any $x\in E-L$,
  so $L_h=(E-L)\cup\{a_1,a_h\}$ for $2\leq h\leq q+1$.  Since $q>2$,
  the line $L_h$ of $N_{j_h}$ then has subsets $A$ with $|A|=3$ and
  $r_M(A)=2$, which is a contradiction, so $\sigma$ is indecomposable.
\end{ex:ppl2}

\section{Chromatic numbers and chromatic polynomials}\label{sec:chi}

After noting some basic properties of the chromatic number and
chromatic polynomial of a polymatroid, we prove our second main
result, Theorem \ref{thm:graphbehind2poly}: the chromatic polynomial
of a $2$-polymatroid is a rational multiple of the chromatic
polynomial of a graph.

Note that $\chi(\rho;x)$ is a polynomial since, when $k$ is a positive
integer, $\chi(\rho;k)$ is the following sum of multinomial
coefficients, which are polynomials in $k$:
\begin{equation}\label{eq:multinomialexpansion}
  \chi(\rho;k) = \sum_{i=\chi(\rho)}^{\rho(E)}\sum_{\{M_s\,:\,s\in[i]\}}
  \binom{k}{a_1,a_2,\ldots,a_h,k-i},
\end{equation}
where the inner sum is over all decompositions of $\rho$ with $i$
matroids, all of positive rank, and $a_1\geq a_2\geq \cdots \geq a_h$
are the multiplicities of the distinct matroids in the decomposition.
If $\rho$ is indecomposable, then the sum is empty.  If
$k<\chi(\rho)<\infty$, then the multinomial coefficients are zero.
The degree of $\chi(\rho;x)$ is the largest number of matroids of
positive rank in a decomposition of $\rho$.  If $\rho$ is the Boolean
polymatroid of a graph $G$ with no isolated vertices, then the degree
of $\chi(\rho;x)$ is the number of vertices of $G$.

The \emph{direct sum} $\rho_1\oplus\rho_2$ of polymatroids $\rho_1$
and $\rho_2$ on disjoint sets $E_1$ and $E_2$ is defined by extending
the definition for matroids: specifically, for
$X\subseteq E_1\cup E_2$,
$$(\rho_1\oplus\rho_2)(X) = \rho_1(X\cap E_1)+\rho_2(X\cap E_2).$$
Polymatroids that are not direct sums of other polymatroids are
\emph{connected}.  It follows that if
$\rho = r_{M_1}+ r_{M_2}+\cdots+ r_{M_k}$, then $\rho$ is disconnected
if and only if there is a set $X$ with
$\emptyset\subsetneq X\subsetneq E$ so that
$M_i= M_i|X\oplus M_i|(E-X)$ for all $i\in [k]$.  It is easy to see
that $\chi(\rho_1\oplus\rho_2) = \max\{\chi(\rho_1),\chi(\rho_2)\}$
and
\begin{equation}\label{eqn:chrodirsum}
  \chi(\rho_1\oplus\rho_2;x) = \chi(\rho_1;x)\cdot\chi(\rho_2;x).
\end{equation}

By equations (\ref{eq:mindelcon}), if $\rho'$ is a minor of $\rho$,
then $\chi(\rho')\leq \chi(\rho)$.  
(In
contrast, for a graph, such as an even cycle, we can have
$\chi(G/e)>\chi(G)$.)

\begin{example}\label{ex:paving}
  Unlike the chromatic polynomial of a graph, $\chi(\rho;x)$ need not
  be monic.
  This follows from an attractive
  observation in \cite[Section 3]{manoel}.  Let $M$ be a paving
  matroid of rank $r$ on $E$ with at least two hyperplanes that are
  cyclic (that is, unions of circuits).  Let $\rho$ be given by the
  decomposition $\{M,U_{r,E}\}$.  No proper, nonempty cyclic flat of
  $M$ contains another, so we can relax any subset of these flats,
  thereby turning all $r$-subsets of the cyclic hyperplanes that we
  relax into bases.  If we partition the set of cyclic hyperplanes of
  $M$ into two sets $\{X_1,X_2\}$ (allowing one to be empty), then the
  corresponding set $\{M_1, M_2\}$ of relaxations of $M$ is a
  decomposition of $\rho$.
  If $M$ has $s$ cyclic hyperplanes,
  then there are $2^{s-1}$ such decompositions $\{M_1,M_2\}$;
  furthermore, they are all of the decompositions of $\rho$ using
  matroids of positive rank.  Thus, $\chi(\rho;x) = 2^{s-1}x(x-1)$,
  which is not monic.  (This idea can be adapted to any matroid whose
  proper, nonempty cyclic flats are incomparable.)
\end{example}

The coefficients of $\chi(\rho;x)$ might not be integers.
For
example, if $M$ is a connected matroid on $E$, then the decompositions
of $\rho=r_M+r_M$ are $\{M,M,U_{0,E},\ldots,U_{0,E}\}$, so
$\chi(\rho;x) = x(x-1)/2$.

The next example shows that $\chi(\rho;x)$ need not be a scalar
multiple of the chromatic polynomial of a graph.

\begin{example}\label{ex:morethanchrompoly}
  Let $\rho$ be the $3$-polymatroid on $E=\{a,b,c,d\}$ in which each
  element has rank three, each pair of elements has rank five, and
  each set of three or four elements has rank six.  In any
  decomposition of $\rho$, each element has rank one in three
  matroids, each pair of elements is parallel in one matroid, and no
  matroid has a coloop.  We now determine all decompositions of $\rho$
  using matroids of positive rank.

  First assume that one of the matroids is $U_{1,E}$.  The other
  matroids in the decomposition have no coloops and no parallel
  elements, and their ranks, which all exceed one, add to five, so
  they are $U_{2,E}$ and $U_{3,E}$.  There is just one such
  decomposition of $\rho$.

  Now assume that the largest parallel class in any of the matroids
  has three elements.  One of the four options is $\{a,b,c\}$, in
  which case $\{a,d\}$, $\{b,d\}$, and $\{c,d\}$ are parallel classes
  in three different matroids.  Also, each of $a$, $b$, and $c$ is a
  non-loop and non-coloop in one more matroid, so the remaining
  matroid must be $U_{2,\{a,b,c\}}\oplus U_{0,\{d\}}$. There are four
  such decompositions of $\rho$, each with five different matroids of
  positive rank.

  If all parallel classes have size two, then each matroid in the
  decomposition of $\rho$ is $U_{1,X}\oplus U_{0,E-X}$ or
  $U_{1,X}\oplus U_{1,E-X}$ for some $X\subseteq E$ with $|X|=2$.
  This accounts for
  \begin{itemize}
  \item one decomposition with six matroids of positive rank,
  \item three decomposition with five matroids of positive rank,
  \item three decomposition with four matroids of positive rank, and
  \item one decomposition with three matroids of positive rank.
  \end{itemize}

  Thus, letting $(x)_i =x(x-1)\cdots(x-i+1)$, we have
  \begin{align*}
    \chi(\rho;k)= &\, (x)_6+7(x)_5+3(x)_4+2(x)_3\\
     =&\,\, x^6 - 8 x^5 + 18 x^4 + 4 x^3 - 49 x^2 + 34 x.
  \end{align*}
  This is not the chromatic polynomial of any graph.
\end{example}

The next result, which stands in contrast to Example
\ref{ex:morethanchrompoly}, supports the observation that
$2$-polymatroids have more in common with graphs than do
$k$-polymatroids for $k>2$.

\begin{thm}\label{thm:graphbehind2poly}
  If $\rho$ is a $2$-polymatroid, then $\chi(\rho;x)=s\cdot\chi(G;x)$
  for some graph $G$ and rational number $s$.
\end{thm}

The paper of Lemos \cite{manoel} provides many tools that we use in
the proof, so we start by outlining his results and explaining how
they give Theorem \ref{thm:graphbehind2poly} for many
$2$-polymatroids.
This reduces the proof to treating
$2$-polymatroids with very special decompositions, which we address in
Lemma \ref{not3}, with the aid of results of Lemos \cite{manoel1} and
Lemos and Mota \cite{LemosMota}.

If $\rho$ is not decomposable, then Theorem \ref{thm:graphbehind2poly}
holds with $s=0$, so we may assume that $\rho$ is decomposable.  By
equation (\ref{eqn:chrodirsum}), we may assume that $\rho$ is
connected.

Lemos \cite{manoel} defines two decompositions of a $2$-polymatroid to
be \emph{equivalent} if one can be obtained from the other by applying
a sequence of the following operations or their inverses:
\begin{enumerate}
\item remove a matroid of rank zero and
\item replace a disconnected matroid $M\oplus N$ by
  $M\oplus U_{0,E(N)}$ and $N\oplus U_{0,E(M)}$.
\end{enumerate}

Assume, first, that all decompositions of $\rho$ are
equivalent. 
Let $D=\{M_i\,:\,i\in[n]\}$ be the unique decomposition of
$\rho$ so that, for all $i\in [n]$, exactly one connected component
$A_i$ of $M_i$ has positive rank.  Let $G$ be the graph on the vertex
set $[n]$ in which $ij$ is an edge if $i\ne j$ and
$A_i\cap A_j\ne \emptyset$.  We can replace
$M_{i_1}, M_{i_2},\ldots,M_{i_t}$ by a direct sum that reverses
operation (2) above if and only if the vertex set
$\{i_1, i_2,\ldots,i_t\}$ is independent in $G$.  Therefore, if
$m_1,m_2,\ldots,m_h$ are the multiplicities of the distinct matroids
in $D$, then $$\chi(\rho;x)=\frac{\chi(G;x)}{m_1!m_2!\cdots m_h!}.$$

Now assume that $\rho$ has inequivalent decompositions.  As defined by
Lemos \cite{manoel}, a pair
$(\{M_i\,:\,i\in[k]\}, \{N_j\,:\,j\in[h]\})$ of decompositions of
$\rho$ on $E$ is \emph{preserving} if for all $X\subseteq E$ with
$|X|>1$, there are as many integers $i \in [k]$ for which $X$ is a
circuit of $M_i$ as integers $j \in [h]$ for which $X$ is a circuit of
$N_j$; otherwise the (necessarily inequivalent) decompositions are
\emph{non-preserving}.  The Boolean $2$-polymatroid of the $3$-cycle
$C_3$ on the edge set $E$ has the non-preserving pair
$(\{U_{1,E},U_{2,E}\},\{U_{1,E-i}\oplus U_{0,i}\,:\,i\in E\})$.
Theorem 1 of \cite{manoel}, stated next as Theorem
\ref{thm:manoelthm1}, shows that, when combined with $2$-sums, this
accounts for all non-preserving pairs of decompositions of connected
$2$-polymatroids.  (Here, $2$-sums with $U_{1,2}$ just relabel points,
so Theorem \ref{thm:manoelthm1} includes $\rho_{C_3}$.)

\begin{thm}\label{thm:manoelthm1}
  Let $\rho$ be a connected $2$-polymatroid on $E$.  Fix a set
  $A=\{a,b,c\}$ with $A\cap E=\emptyset$. 
  A pair of
  decompositions of $\rho$ is non-preserving if and only if, for some
  partition $\{E_a, E_b, E_c \}$ of $E$ and connected matroids $P_a$,
  $P_b$, and $P_c$ with $E(P_i ) = E_i \cup i$ for each $i \in A$, the
  decompositions are
  $\{U_{1, A} \oplus_2 P_a \oplus_2 P_b \oplus_2 P_c, \,\, U_{2, A}
  \oplus_2 P_a \oplus_2 P_b \oplus_2 P_c\}$ and
  $\{(U_{1,\{i,j\}} \oplus_2 P_i \oplus_2 P_j) \oplus U_{0,A-\{i,j\}}
  \,:\, \{i,j\}\subsetneq A\}$.
\end{thm}

It follows that if $\rho$ has a pair of non-preserving decompositions,
then, up to adjoining rank-$0$ matroids, the decompositions of $\rho$
are those given above, and so
$$\chi(\rho;x)=x(x-1)(x-2)+x(x-1)=x(x-1)^2,$$
which is the chromatic polynomial of a path on three vertices.

We now focus on connected $2$-polymatroids that have inequivalent
decompositions but all pairs of decompositions are preserving.
Theorem 2 of \cite{manoel}, stated next as Theorem
\ref{thm:manoelthm2}, treats such $2$-polymatroids.

\begin{thm}\label{thm:manoelthm2}
  If $\rho$ has inequivalent preserving decompositions, then each
  decomposition of $\rho$ is equivalent to one that contains exactly
  two matroids.
\end{thm}

Corollary 1 of \cite{manoel}, stated next as Theorem
\ref{thm:manoelthm3}, further limits the options.

\begin{thm}\label{thm:manoelthm3}
  If a connected $2$-polymatroid $\rho$ has a decomposition in which
  at least four matroids have positive rank, then all decompositions
  of $\rho$ are equivalent.
\end{thm}

As shown above, in this case the conclusion of Theorem
\ref{thm:graphbehind2poly} holds.  Two cases remain: $\rho$ has
inequivalent decompositions and either
\begin{enumerate}
\item[(1)] all decompositions have exactly two matroids of positive
  rank, or
\item[(2)] some decomposition has a connected matroid and one with
  exactly two connected components.
\end{enumerate}
Option (1) was illustrated in Example \ref{ex:paving}. As in that
example, the chromatic polynomial of $\rho$ is a rational multiple of
the chromatic polynomial $x(x-1)$ of the complete graph $K_2$.

We turn to option (2).  In \cite[(2.12)]{manoel1}, Lemos proves that
when option (2) applies, only one decomposition of $\rho$ has a
disconnected matroid.  Building on results in \cite{LemosMota}, below
we complete the proof of Theorem \ref{thm:graphbehind2poly} by
analyzing how many decompositions such a connected $2$-polymatroid
has.

By \cite[Lemma 2.3]{LemosMota}, if $M_1$, $M_2$, $N_1$, $N_2$ are
matroids on $E$ with $r_{M_1}+r_{M_2} = r_{N_1}+r_{N_2}$ and this
polymatroid is connected, then
$\{r_{M_1}(A),r_{M_2}(A)\} = \{r_{N_1}(A),r_{N_2}(A)\}$ for each
subset $A$ of $E$.  Thus, if $M_1\ne M_2$, then $N_1\ne N_2$.  Also,
given $M_1$ and $M_2$, one can find all pairs of matroids $N_1$ and
$N_2$ for which $r_{M_1}+r_{M_2} = r_{N_1}+r_{N_2}$ by ``mixing'' in
the sense that there is a subset $\mathscr{R}$ of $2^E$ for which, for
$i\in[2]$,
\begin{equation}\label{eq:defmix}
r_{N_i}(A) =
\begin{cases}
  r_{M_i}(A) & \text{if } A\in\mathscr{R},\\
  r_{M_{3-i}} (A) & \text{if } A\notin\mathscr{R}.
\end{cases}
\end{equation}
This leads to the \emph{mixing graph} $G_{M_1,M_2}$ of $M_1$ and $M_2$
that is defined in \cite[Section 3]{LemosMota}:  
its vertex
set is $\mathscr{X}=\{A\subseteq E: r_{M_1}(A)\neq r_{M_2}(A)\}$, and
$AB$ is an edge of $G_{M_1,M_2}$ if and only if either
\begin{enumerate}
\item[(i)] $A\subsetneq B$ or $B\subsetneq A$, or
\item[(ii)] $A\cap B\notin \mathscr{X},A\cup B\notin \mathscr{X}$, and
  $|A\btu B|=2$.
\end{enumerate}
Let $\mathscr{R}$ be any union of the vertex sets of connected
components of $G_{M_1,M_2}$, and define $N_1$ and $N_2$ by equation
(\ref{eq:defmix}).  Lemos and Mota \cite[Lemma 3.1]{LemosMota} show
that $r_{N_1}$ and $r_{N_2}$ are the rank functions of matroids.  By
construction, $r_{M_1}+r_{M_2} = r_{N_1}+r_{N_2}$.  In \cite[Theorem
4.1]{LemosMota}, they show all matroids $N_1$ and $N_2$ with
$r_{M_1}+r_{M_2} = r_{N_1}+r_{N_2}$ are obtained in this way.  Note
that $\mathscr{X}-\mathscr{R}$ gives the same decomposition as
$\mathscr{R}$, but $N_1$ and $N_2$ are switched.

To address option (2), we prove in Lemma \ref{not3} below that if
$M_1$ is connected and $M_2$ has two connected components, $P$ and
$Q$, then the graph $G_{M_1,M_2}$ has at most two components.  Since
$\rho$ has inequivalent decompositions, $G_{M_1,M_2}$ has two
components.  That completes the proof of Theorem
\ref{thm:graphbehind2poly} since then $\rho$ has just two inequivalent
decompositions: if $\mathscr{R}$ is $\emptyset$ or $\mathscr{X}$,
equation (\ref{eq:defmix}) gives $M_1$ and $M_2$, otherwise it gives
two connected matroids, $N_1$ and $N_2$.  Thus, the decompositions of
$\rho$ into matroids of positive rank are $\{M_1,M_2\}$,
$\{M_1,(M_2|P)\oplus U_{0,Q}, (M_2|Q)\oplus U_{0,P}\}$, and
$\{N_1,N_2\}$, so
$$\chi(\rho;x)=2x(x-1)+x(x-1)(x-2) = x^2(x-1),$$  which is the
chromatic polynomial of a graph.

Below we shorten the notation, using $G$ to denote $G_{M_1,M_2}$, and
$r_1$ and $r_2$ to denote the rank functions of $M_1$ and $M_2$,
respectively.

\begin{lemma}\label{2symdif}
  If $A,B\in\mathscr{X}$ and $|A\btu B|\leq 2$, then $G$ contains an
  $A,B$-path.
\end{lemma}

\begin{proof}
  If $A=B$ or $AB$ is an edge, then the conclusion holds; otherwise
  $|A\btu B|=2$ and either $A\cup B$ or $A\cap B$ is in $\mathscr{X}$,
  so either $A(A\cup B)B$ or $A(A\cap B)B$ is a path in $G$.
\end{proof}

We will augment $G$ by adding the edge $AB$ for each pair of vertices
$A,B\in\mathscr{X}$ for which $|A\btu B|= 2$.  By the lemma, this does
not change the components of $G$, and it simplifies writing some of
the paths that we use below.

\begin{lemma}\label{not3}
  If $M_1$ is connected and $M_2$ has exactly two components, $P$ and
  $Q$, then $G$ has at most two components.
\end{lemma}

\begin{proof}
  Suppose, to the contrary, that $G$ has at least three components.
  Let
  $$V_1=\{Z\in\mathscr{X}\,:\,r_1(Z)>r_2(Z)\} \quad\text{ and }
  \quad V_2=\{Z\in\mathscr{X}\,:\,r_2(Z)>r_1(Z)\}.$$ Either $V_1$ or
  $V_2$ must contain vertices from different components of $G$.  Take
  $X,Y\in\mathscr{X}$ that are in different components of $G$ but in
  the same set $V_i$, for some $i\in[2]$, and so that
  $|X\btu Y|,|X|$, and $|Y|$ are minimized, in that order.

  Extend a basis of $M_i|X\cap Y$ to bases $B_X$ of $M_i|X$ and $B_Y$
  of $M_i|Y$.  Then $B_X,B_Y\in V_i$.  By minimality, $X=B_X$ and
  $Y=B_Y$.  For $y\in Y-X$, since $|(X\cup y)\btu Y|<|X\btu Y|$,
  minimality gives $X\cup y\not\in V_i$.  Thus, for all $y\in Y-X$,
  $$r_i(X)=r_i(X\cup y)=r_{3-i}(X\cup y)=r_{3-i}(X)+1.$$
  Thus, $X$ is a basis of $M_i|X\cup Y$ and a flat of
  $M_{3-i}|X\cup Y$.  The same holds for $Y$.  Now
  $r_{3-i}(X\cup Y)= r_i(X\cup Y)$ since $X(X\cup Y)Y$ is not a path
  in $G$, so $X$ and $Y$ are hyperplanes of $M_{3-i}|X\cup Y$, which
  has rank $|X|=|Y|$.  Thus, each of $M_{3-i}|X$ and $M_{3-i}|Y$
  contains exactly one circuit, say $C_X$ and $C_Y$, respectively.  If
  $x\in X-(C_X\cup C_Y)$, then $X-x$ and $Y-x$ would both be in $V_i$
  and $|(X-x)\btu (Y-x)|\leq |X\btu Y|$, contrary to the minimality of
  $|X\btu Y|$ or $|X|$.  Thus, $X\subseteq C_X\cup C_Y$, and by
  symmetry, $X\cup Y=C_X\cup C_Y$.  We summarize this information in
  the following sublemma.

  \begin{sublemma}\label{bch}
    The following assertions hold:
    \begin{enumerate}
    \item $X$ and $Y$ are bases of $M_i|X\cup Y$,
    \item $X$ and $Y$ are hyperplanes of $M_{3-i}|X\cup Y$,
    \item $C_X$ and $C_Y$ are the unique circuits of $M_{3-i}|X$ and
      $M_{3-i}|Y$, respectively, and
    \item $X\cup Y= C_X\cup C_Y$, so $X-Y\subseteq C_X$ and
      $Y-X\subseteq C_Y$.
    \end{enumerate}
  \end{sublemma}

  Now $XEY$ is not a path in $G$, so $E\notin\mathscr{X}$, and so
  $r_1(E)=r_2(E)$.  Therefore,
  \[r_2(P)+r_2(Q)=r_2(E)=r_1(E)<r_1(P)+r_1(Q).\] Thus, either $P$ or
  $Q$ is in $V_1$.  By symmetry, we may assume that $P\in V_1$.  Next
  we show that
  \begin{sublemma}\label{iis1}
    $i=1$.
  \end{sublemma}
  Suppose instead that $i=2$.  Let $X_P=X\cap P$, and likewise for
  $X_Q$, $Y_P$, and $Y_Q$.

  Since $M_2=M_2|P \oplus M_2|Q$, by~\ref{bch}(1) both $X_P$ and $Y_P$
  are bases of $M_2|X_P\cup Y_P$, as are $X_Q$ and $Y_Q$ for
  $M_2|X_Q\cup Y_Q$.  Thus, $|X|=|Y|$, $|X_P|=|Y_P|$, and
  $|X_Q|=|Y_Q|$.  Therefore $|X_R-Y_R|= |Y_R-X_R|$ for $R=P$ and
  $R=Q$.  Thus, since $X\ne Y$, for either $R=P$ or $R=Q$, there are
  elements $x_R\in X_R-Y_R$ and $y_R\in Y_R-X_R$.  Since $X_R$ and
  $Y_R$ are bases of $M_2|X_R\cup Y_R$, both $X_R\cup y_R$ and
  $Y_R\cup x_R$ are dependent in $M_2$.
  
  If $X_R$ and $Y_R$ were independent in $M_1$, then $X_R\cup y_R$ and
  $Y_R\cup x_R$ would be independent in $M_1$ by~\ref{bch}(2).  Then
  $X_R\cup y_R$ and $Y_R\cup x_R$ would be in $V_1$, but this
  contradicts minimality.  Thus, either $C_X\subseteq X_R$ or
  $C_Y\subseteq Y_R$ by~\ref{bch}(3). By symmetry, we may assume that
  $C_X\subseteq X_R$.  Then $X-Y\subseteq C_X-Y\subseteq X_R-Y_R$, so
  $X-Y=X_R-Y_R$.  Thus $|Y_R-X_R| = |X_R-Y_R|= |X-Y|= |Y-X|$, so
  $X\btu Y=X_R\btu Y_R$.

  We claim that $Y_R$ is dependent in $M_1$.  Suppose instead that
  $Y_R$ is independent in $M_1$, so $Y_R\cup x_R\in V_1$.  If $R=P$,
  then the set $X_Q = Y_Q$ is not empty since $XPY$ is not a path in
  $G$.  Since $C_X\subseteq X_P\subseteq P$, we have
  $Y_Q\subseteq C_Y$ by~\ref{bch}(4).  Thus, if $e \in Y_Q$, then
  $(Y\cup x_P)-e$ is independent in $M_1$.  Since $(Y\cup x_P)-e$
  contains $Y_P\cup x_P$, it is dependent in $M_2$; thus,
  $(Y\cup x_P)-e \in V_1$.  This gives the contradiction that
  $XC_XP(Y_P\cup x_P)((Y\cup x_P)-e)Y$ is a path in $G$ (the last
  adjacency is by Lemma~\ref{2symdif} and the note after it; here and
  in similar paths below, we omit mentioning the obvious adjustments
  when vertices are repeated, for instance, if $X=C_X$).  Thus, $R=Q$.
  Therefore $X_P=Y_P$.  Let $P'$ be a minimum-sized set in
  $\mathscr{X}$ for which $X_P\subseteq P'\subseteq P$.  If
  $P'\in V_2$, then, since $M_2=M_2|P\oplus M_2|Q$ and since $X_Q$ and
  $Y_Q$ are bases of $M_2|X_Q\cup Y_Q$, both $P'\cup X_Q$ and
  $P'\cup Y_Q$ are in $V_2$, but this gives the contradiction that
  $X(P'\cup X_Q)P'(P'\cup Y_Q)Y$ is a path in $G$.  Thus $P'\in V_1$.
  Now $X_P$ is independent in $M_2$, as well as in $M_1$ since
  $C_X\subseteq X_Q$.  Thus, by minimality, $P'$ is independent in
  $M_1$ and contains a unique circuit of $M_2$, and that circuit
  contains $P'-X_P$.  Thus, for any element $p\in P'-X_P$, both
  $(P'-p)\cup X_Q$ and $(P'-p)\cup Y_Q$ are independent in $M_2$.
  Also, $(P'-p)\cup X_Q$ and $(P'-p)\cup Y_Q$ contain $X$ and $Y$,
  respectively, and so are dependent in $M_1$; thus, these sets are in
  $V_2$.  Furthermore, $r_1(P'\cup X_Q)= r_1((P'\cup X_Q)-x_Q)$ since
  $x_Q\in X_Q-Y_Q\subseteq C_X$, but
  $r_2(P'\cup X_Q)> r_2((P'\cup X_Q)-x_Q)$ since
  $M_2=M_2|P\oplus M_2|Q$ and $X_Q$ is independent in $M_2|Q$.  Hence
  either $P'\cup X_Q$ or $ (P'\cup X_Q)-x_Q$ is in $\mathscr{X}$.
  Likewise, either $P'\cup Y_Q$ or $(P'\cup Y_Q)-y_Q$ is in
  $\mathscr{X}$.  Thus, one of the following is a path in $G$ by
  Lemma~\ref{2symdif}.
  \begin{itemize}
  \item $X((P'-p)\cup X_Q)(P'\cup X_Q)P'(P'\cup Y_Q)((P'-p)\cup Y_Q)Y$
  \item $X((P'-p)\cup X_Q)(P'\cup X_Q)P'((P'\cup Y_Q)-y_Q)((P'-p)\cup
    Y_Q)Y$
  \item $X((P'-p)\cup X_Q)((P'\cup X_Q)-x_Q)P'(P'\cup Y_Q)((P'-p)\cup
    Y_Q)Y$
  \item $X((P'-p)\cup X_Q)((P'\cup X_Q)-x_Q)P'((P'\cup
    Y_Q)-y_Q)((P'-p)\cup Y_Q)Y$
  \end{itemize}
  This contradiction proves that $Y_R$ is dependent in $M_1$.

  Thus, $C_Y\subseteq Y_R$.  Since $C_X\subseteq X_R$, we have
  $X\cup Y=C_X\cup C_Y\subseteq X_R\cup Y_R$, so $X=X_R$ and $Y=Y_R$.
  Now $XPY$ is not a path in $G$, so $R\ne P$.  Thus, $X=X_Q$ and
  $Y=Y_Q$, so $X\cup Y\subseteq Q$.

  We claim that $G$ has an $X,P$-path.  Since $P\in V_1$, some subset
  $P'$ of $P$ is independent in $M_1$ and is a circuit of $M_2$.  If
  $P'\cup X\in\mathscr{X}$, then $PP'(P'\cup X)X$ is a path in $G$, as
  claimed, so assume that $P'\cup X\notin\mathscr{X}$.  Since $P'$ and
  $X$ are subsets of different components of $M_2$, and $X$ is
  independent in $M_2$, the only circuit of $M_2|P'\cup X$ is $P'$.
  Thus, since $P'\cup X\notin\mathscr{X}$,
  $$r_1(P'\cup X)=r_2(P'\cup X) = |P'|+|X|-1.$$ It follows that
  $P'\cup X$ contains exactly one circuit of $M_1$, and that is $C_X$
  since $C_X\subseteq X$.  Thus, if $p\in P'$, then
  $$r_1((P'-p)\cup X)<r_1(P'\cup X)=r_2(P'\cup X) =r_2((P'-p)\cup
  X),$$ and if $x\in C_X$, then
  $$r_1(P'\cup (X-x))=r_1(P'\cup X)=r_2(P'\cup X) >r_2(P'\cup
  (X-x)).$$ Thus, $PP'(P'\cup (X-x))((P'-p)\cup X)X$ is a path in $G$
  by Lemma~\ref{2symdif}, as desired. 
    By symmetry, $G$ has a $Y,P$-path. This
  contradiction to $X$ and $Y$ being in different components of $G$
  completes the proof of~\ref{iis1}.

  \begin{sublemma}\label{notPnotQ}
    We have $X\cup Y\nsubseteq P$ and $X\cup Y\nsubseteq Q$.
  \end{sublemma}

  Indeed, $X\cup Y\nsubseteq P$ since $XPY$ is not a path in $G$.
  Also, $X(X\cup Y\cup P)Y$ is not a path in $G$, so
  $X\cup Y\cup P\notin\mathscr{X}$.  Hence
  $r_1(X\cup Y\cup P)=r_2(X\cup Y\cup P)$.  Conclusion~\ref{bch}(1)
  gives $r_1(X)=r_1(Y)=r_1(X\cup Y)$, while
  $r_2(X)=r_2(Y)=r_2(X\cup Y)-1$ by~\ref{bch}(2).  If $X\cup Y$ were
  contained in the component $Q$ of $M_2$, then we would have
  \[r_2(X\cup P)=r_2(Y\cup P)<r_2(X\cup Y\cup P)=r_1(X\cup Y\cup
    P)=r_1(X\cup P)=r_1(Y\cup P),\] so
  $X\cup P,Y\cup P\in\mathscr{X}$, and $X(X\cup P)P(Y\cup P)Y$ would
  be a path in $G$; this contradiction completes the proof of
  \ref{notPnotQ}.
  
  By~\ref{bch}(2), the sets $X-Y$ and $Y-X$ are cocircuits of
  $M_2|X\cup Y$, so each is a subset of either $P$ or $Q$.  Thus,
  $X_R-Y_R=\emptyset$ and $Y_S-X_S=\emptyset$ for some
  $R,S\in\{P,Q\}$.  We claim that $R\ne S$.  Suppose not, so
  $X_R=Y_R$.  Submodularity gives
  \begin{equation}\label{eq:submodrspq}
    r_1(X_P)+r_1(X_Q)\geq r_1(X)>r_2(X)=r_2(X_P)+r_2(X_Q)
  \end{equation}
  and, similarly, $ r_1(Y_P)+r_1(Y_Q)>r_2(Y_P)+r_2(Y_Q)$.  If $R=P$,
  then $X_P\notin \mathscr{X}$ since $XX_PY$ is not a path in $G$, and
  so inequality (\ref{eq:submodrspq}) gives $X_Q\in V_1$; similarly,
  $Y_Q\in V_1$.  Likewise, if $R=Q$, then $X_P\in V_1$ and
  $Y_P \in V_1$.  By \ref{notPnotQ}, both options contradict
  minimality, so, as claimed, $R\ne S$. 

  The order in the minimality assumption is the only asymmetry between
  $X$ and $Y$.  That plays no role below, so, of the two ways to
  have $R\ne S$, 
  we may assume that
  \begin{sublemma}\label{rspqincl}
    $Y_P\subsetneq X_P$ and $X_Q\subsetneq Y_Q$.  Thus, $X_P=C_X$ and
    $Y_Q=C_Y$ by \emph{\ref{bch}(4)}. 
  \end{sublemma}

  Since $G$ has at least three components, take $W\in \mathscr{X}$ so
  that $G$ has no $X,W$-path and no $Y,W$-path, and $|Y\btu W|$ and
  $|W|$ are minimized, in that order.  Then $W\in V_j$ for some
  $j\in[2]$.  Since $Y(Y\cap W)W$ is not a path,
  $Y\cap W\notin\mathscr{X}$. Likewise, $Y\cup W\notin\mathscr{X}$.
  Since $Y\cap W$ is independent in $M_1$, it is also independent in
  $M_2$.  The minimality assumption now implies that $W$ is
  independent in $M_j$ and has nullity one in $M_{3-j}$.  Also, $W-Y$
  is a subset of the unique circuit $C_W$ of $M_{3-j}|W$.  By the
  minimality of $|Y\btu W|$, if $y\in Y-W$, then
  $W\cup y\not\in \mathscr{X}$, so
  $r_j(W)=r_j(W\cup y)= r_{3-j}(W\cup y)=r_{3-j}(W)+1$.  Hence $W$ is
  a basis of $M_j|W\cup Y$ and, since $Y\cup W\notin\mathscr{X}$, a
  hyperplane of $M_{3-j}|W\cup Y$.

  We show that
  \begin{sublemma}\label{ywyq}
    $Y\btu W\subseteq Q$.  
  \end{sublemma}

  First note that $C_Y\not\subseteq Y\cap W$ since $C_Y$ is a circuit
  of $M_2$ while $Y\cap W$ is independent.  Thus,
  $C_Y\cap(Y-W)\ne\emptyset$, so $Q\cap (Y-W) \ne\emptyset$ by
  \ref{rspqincl}.

  Assume that $j=1$. The circuit $C_W$ is then a subset of $P$ or $Q$.
  Now $WC_WPC_XX$ is not a path in $G$, so $C_W\not\subseteq P$.
  Thus, $W-Y\subseteq C_W\subseteq Q$.  The cocircuit $Y-W$ of
  $M_2|Y\cup W$ is a subset of $P$ or $Q$. Since
  $Q\cap (Y-W) \ne\emptyset$, we get $Y\btu W\subseteq Q$.

  Now take $j=2$.  Fix $y\in Y-W$.  Since $W$ is a basis of
  $M_2|W\cup Y$, the set $W\cup y$ contains a unique circuit $C_y$ of
  $M_2$.  Now $W$ is a hyperplane of $M_1|W\cup Y$ and $C_W$ is the
  unique circuit of $M_1|W$, so $C_W$ is the unique circuit of
  $M_1|W\cup y$.  If there were an element $w$ in $(W-Y)-C_y$, then
  $(W-w)\cup y$ would contain the circuit $C_y$ of $M_2$, but contain
  no circuit of $M_1$ since $W-Y\subseteq C_W$, so
  $(W-w)\cup y\in V_1$.  This contradicts the minimality of
  $|Y\btu W|$ since there would be a path from $(W-w)\cup y$ to $W$ by
  Lemma~\ref{2symdif}.  Thus, $(W-Y)\cup y\subseteq C_y$.  The circuit
  $C_y$ of $M_2$ is a subset of either $P$ or $Q$, and so
  $(W-Y)\cup y$ is a subset of either $P$ or $Q$.  This holds for all
  $y\in Y-W$, so $(W-Y)\cup (Y-W)$ is a subset of either $P$ or $Q$.
  Now \ref{ywyq} follows since $(Y-W)\cap Q \ne\emptyset$.

  Now $Y\btu W\subseteq Q$ by \ref{ywyq}, so $(Y\cup W)\cap P= Y_P$.
  Now $X_P\not\subseteq \cl _2(Y_P)$ by \ref{bch}(2) and
  \ref{rspqincl}.  Thus, $X_P\not\subseteq \cl _2((Y\cup W)\cap P)$,
  and so $X_P\nsubseteq \cl _2 (Y\cup W)$ since $P$ is a component of
  $M_2$.

  Finally, $X_P\subseteq \cl_1(Y)$ by \ref{bch}(1).  Now
  $Y\cup W\cup X_P\notin\mathscr{X}$ since $Y(Y\cup W\cup X_P)W$ is
  not a path in $G$.  We now get $Y\cup W\in\mathscr{X}$ since, using
  $X_P\nsubseteq \cl _2 (Y\cup W)$, we have
  \[r_2(Y\cup W)<r_2(Y\cup W\cup X_P)=r_1(Y\cup W\cup X_P)=r_1(Y\cup
    W).\] Thus, $Y(Y\cup W)W$ is a path in $G$.  This contradiction
  completes the proof.
\end{proof}

\section{Duals and chromatic numbers}\label{sec:duality}

The \emph{$i$-dual} of an $i$-polymatroid $\rho$ on $E$ is the
$i$-polymatroid $\rho^*$ on $E$ that is given by 
\begin{equation}\label{def:kdual}
  \rho^*(X) = i\,|X|-\rho(E)+\rho(E-X)
\end{equation}
for $X\subseteq E$.  In this section, we show that under certain
conditions, an $i$-polymatroid and its $i$-dual have the same
chromatic number or the same chromatic polynomial.

Before treating the basic result in this section, Theorem
\ref{thm:gendualresult}, we note that if $i=1$, then $\rho$ is a
matroid and equation (\ref{def:kdual}) gives the dual matroid.  It is
easy to check that $(\rho^*)^* = \rho$, and that
$(\rho_{\del A})^* = (\rho^*)_{/A}$ and
$(\rho_{/A})^* = (\rho^*)_{\del A}$.  The $i$-dual depends on $i$.
That is relevant since an $i$-polymatroid is a $j$-polymatroid for
each $j\geq i$ and so has a $j$-dual.  The $i$-dual $\rho^{*_i}$ and
$j$-dual $\rho^{*_j}$ are related by
$\rho^{*_j}(X)= \rho^{*_i}(X)+(j-i)|X|$ for all $X\subseteq E$.

\begin{thm}\label{thm:gendualresult}
  Let $\rho^*$ be the $i$-dual of an $i$-polymatroid $\rho$ on
  $E$.
  If $\{M_s\,:\,s\in[i]\}$ is a
  decomposition of $\rho$, then $\{M^*_s\,:\,s\in[i]\}$ is a
  decomposition of $\rho^*$.  More generally, if $k\geq i$ and
  $\{M_s\,:\,s\in[k]\}$ is a decomposition of $\rho$, then $\rho^*$
  has a decomposition $\{M'_s\,:\,s\in[k]\}$ where, for all
  $s\in [k]$, deleting the loops and coloops from $M'_s$ gives the
  dual of $M_s$ with its loops and coloops deleted.  Thus, if
  $\min\bigl(\chi(\rho),\chi(\rho^*)\bigr)\geq i$, then
  $\chi(\rho)=\chi(\rho^*)$.
\end{thm}

\begin{proof}
  Since $\rho(e)\leq i$ for all $e\in E$, there are at most $i$
  integers $s\in[k]$ with $r_{M_s}(e)=1$.  Thus, since $k\geq i$,
  there are subsets $E_1,E_2,\ldots,E_k$ of $E$ such that (i) if
  $r_{M_s}(e)=1$, then $e\in E_s$, and (ii) each element of $E$ is in
  exactly $i$ of the sets $E_1,E_2,\ldots,E_k$.  For $X\subseteq E$,
  \begin{align*}
    \rho^*(X) = &\,\, i|X|-\rho(E)+\rho(E-X)\\
    = &\,\, \sum_{s=1}^k |E_s\cap X|-\sum_{s=1}^k
        r_{M_s}(E_s)+\sum_{s=1}^kr_{M_s}\bigl(E_s-(X\cap E_s)\bigr)\\
    = &\,\, \sum_{s=1}^k \Bigl(|E_s\cap X|-
        r_{M_s}(E_s)+r_{M_s}\bigl(E_s-(X\cap E_s)\bigr)\Bigr).
  \end{align*}
  The summand in the last line is the rank of $E_s\cap X$ in the dual
  of the restriction $M_s|E_s$, which is the rank of $X$ in the direct
  sum, $(M_s|E_s)^*\oplus U_{0,E-E_s}$.  This gives the claimed
  decomposition of $\rho^*$.  If $k=i$, then $E_s=E$ for all
  $s\in[i]$, so $M'_s=M^*_s$. Finally, when
  $\min\bigl(\chi(\rho),\chi(\rho^*)\bigr)\geq i$, what we proved
  applies to all decompositions of $\rho$ and $\rho^*$, so, since
  $(\rho^*)^*=\rho$, the last assertion follows.
\end{proof}

To see that, in the last assertion in Theorem \ref{thm:gendualresult},
the hypothesis is needed, let $E$ be a set with $|E|\geq 3$.
Let $\rho$ be the $2$-polymatroid on $E$ with $\rho(e)=2$ for all
$e\in E$ and $\rho(X)=3$ for all $X\subseteq E$ with $|X|>1$.  The
decomposition $\{U_{1,E},U_{2,E}\}$ of $\rho$ shows that
$\chi(\rho)=2$.  Fix $i$ with $i\geq 3$. In the $i$-dual,
$\rho^*(e)=i$ for all $e\in E$, so $\chi(\rho^*)\geq i$.

By Theorem \ref{thm:gendualresult}, if $i\leq j\leq k$ and an
$i$-polymatroid is in $\mathcal{D}_k$, then its $j$-dual is in
$\mathcal{D}_k$.  In particular, $\mathcal{D}_k$ is
$k$-dual-closed. The next result follows. 

\begin{cor}\label{cor:dualext}
  For $i\leq k$, if $\rho$ is an $i$-polymatroid that is an
  excluded-minor for $\mathcal{D}_k$, then, for each $j$ with
  $i\leq j\leq k$, the $j$-dual of $\rho$ is also an excluded-minor
  for $\mathcal{D}_k$, as is the $j$-polymatroid $\rho^j$ that is
  given by $\rho^j(X) = \rho(X)+(j-i)|X|$.
\end{cor}

The second conclusion uses the first, applied to $\rho^*$, and the
remarks before Theorem \ref{thm:gendualresult}.  Thus, many excluded
minors that we identify yield families of excluded minors.

\begin{cor}\label{cor:dualbijection}
  Let $\rho^*$ be the $i$-dual of an $i$-polymatroid $\rho$ on $E$.
  Assume that for all $e\in E$,
  \begin{itemize}
  \item[(1)] $\rho(E-e)=\rho(E)$; equivalently, $\rho^*(e)=i$, and
  \item[(2)] $\rho^*(E-e)=\rho^*(E)$; equivalently, $\rho(e)=i$.
  \end{itemize}
  There is a bijection $\phi:\Delta_\rho^k\to\Delta_{\rho^*}^k$,
  namely,
  $\phi\bigl((M_1,M_2,\ldots,M_k)\bigr)=(M'_1,M'_2,\ldots,M'_k)$
  where, letting $L_s$ be the set of loops of $M_s$, we have
  $M'_s=(M_s\del L_s)^*\oplus U_{0,L_s}$.
  Thus, $\chi(\rho;x)=\chi(\rho^*;x)$.
\end{cor}

\begin{proof}
  By the assumptions, $\min\bigl(\chi(\rho),\chi(\rho^*)\bigr)\geq i$.
  The result now follows from Theorem \ref{thm:gendualresult} and its
  proof.
\end{proof}

\begin{ex:bpm2}
  The $2$-dual of the Boolean polymatroid $\rho$ of a graph $G=(V,E)$
  with no isolated vertices is given by 
  \begin{equation}\label{eq:interpretgraph2dual}
    \rho^*(A) = 2|A|-\rho(E)+\rho(E-A) =2|A|-\bigl(|V|-\rho(E-A)\bigr),
  \end{equation}
  for $A\subseteq E$.  The difference $|V|-\rho(E-A)$ is the number of
  vertices that are incident only with edges in $A$.  This comes under
  the construction in Section \ref{sec:genboolh} if $G$ has no loops
  and all vertices of $G$ have degree at least two: for the set $X_i$
  of edges that are incident with vertex $v_i$ of $G$, take
  $t_i = |X_i|-1$.  The sum in equation (\ref{eq:genbpmhrk}) is then
  $$\sum_{i=1}^n\min\{|A\cap X_i|,|X_i|-1\},$$
  and each term is $|A\cap X_i|$ unless all edges incident with $v_i$
  are in $A$, in which case the term is reduced by $1$.  Thus, the sum
  is $2|A|-\bigl(|V|-\rho(E-A)\bigr)$ since each edge is in two sets
  $X_i$.  The equality $\rho(E-e)=\rho(E)$ in condition (1) of
  Corollary \ref{cor:dualbijection} holds when neither endvertex of
  $e$ has degree $1$; the equality $\rho(e)=2$ in condition (2)
  excludes loops.  Thus, Corollary \ref{cor:dualbijection} applies to
  Boolean polymatroids of graphs that have no loops and no vertices of
  degree less than two.
\end{ex:bpm2}

\section{$k$-quotient polymatroids}\label{sec:quo}

In this section, we consider \emph{$k$-quotient polymatroids}, that
is, polymatroids of the form $\rho = r_{M_1}+ r_{M_2}+\cdots+ r_{M_k}$
where $k\geq 2$ and each $M_{i+1}$, for $i\in[k-1]$, is a quotient of
(perhaps equal to) $M_i$. 
 We start by recalling the background on
quotients that we need.

For matroids $Q$ and $L$ on a set $E$, the matroid $Q$ is a
\emph{quotient} of $L$, or $L$ is a \emph{lift} of $Q$, if there is
some matroid $M$ and set $A\subseteq E(M)$ for which $L=M\del A$ and
$Q=M/A$. It follows from the definition and from properties of minors
that if $Q$ is a quotient of $L$ and if $X$ is a subset of their
common ground set, then $Q\del X$ is a quotient of $L\del X$, and
likewise for contractions.  We will use the next lemma
\cite[Proposition 7.4.7]{Brylawski}.
  
\begin{lemma}\label{lem:equivquo}
  For matroids $Q$ and $L$ on $E$, the following are equivalent:
  \begin{enumerate}
  \item $Q$ is a quotient of $L$;
  \item $L^*$ is a quotient of $Q^*$;
  \item the difference $r_L-r_Q$ is non-decreasing, that is, for all
    sets $X\subseteq Y\subseteq E$, we have
    $r_L(X)-r_Q(X)\leq r_L(Y)-r_Q(Y)$, or, equivalently,
    \begin{equation}\label{eq:quotientrankineq} 
      r_Q(X\cup e)-r_Q(X)\leq r_L(X\cup e)-r_L(X)
    \end{equation}
    for all $X\subsetneq E$ and $e\in E-X$.
  \end{enumerate}
\end{lemma}
  
It follows that the class of $k$-quotient polymatroids, which we
denote by $\mathcal{Q}_k$, is closed under taking minors and
$k$-duals.  A subclass of $\mathcal{Q}_2$, motivated by \cite{GWRota},
was introduced by D.~Chun in \cite{deb} by imposing the additional
condition that $r(M_1)-r(M_2)\leq 1$.

From inequality (\ref{eq:quotientrankineq}), it follows that for
$\rho\in \mathcal{Q}_k$, we obtain the rank functions $r_{M_i}$
recursively from $\rho$ as follows: $r_{M_i}(\emptyset) = 0$ and, for
$X\subsetneq E$ and $y\in E-X$,
\begin{equation}\label{eq:reconr}
  r_{M_i}(X\cup y) =
  \begin{cases} r_{M_i}(X)+1, & \text{if } \rho(X\cup y)\geq \rho(X)+i,\\
    r_{M_i}(X), & \text{otherwise.}
  \end{cases}
\end{equation} 
Thus, each polymatroid $\rho$ in $\mathcal{Q}_k$ has exactly one
decomposition $\{M_i\,:\,i\in[k]\}$ for which each $M_{i+1}$, for
$i\in[k-1]$, is a quotient of $M_i$.  There may be other
decompositions of $\rho$ that do not satisfy this condition.
For instance, the
Boolean $2$-polymatroid of the $3$-cycle on the edge set $E$ has two
decompositions, $\{U_{2,E},U_{1,E}\}$ and
$\{U_{1,E-x}\oplus U_{0,x}\,:\,x\in E\}$, and only the former
satisfies the quotient condition.

We now give the excluded minors for $\mathcal{Q}_k$, which are
illustrated in Figure \ref{fig:rankincr}.

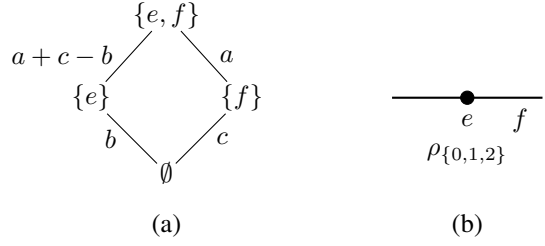
\begin{figure}
  \centering
  \begin{tikzpicture}[scale=1]
  \node[inner sep = 0.3mm] (b) at (1,-0.7) {$\emptyset$}; 
  \node[inner sep = 0.3mm] (l) at (0,0.3) {$\{e\}$}; 
  \node[inner sep = 0.3mm] (r) at (2,0.3) {$\{f\}$}; 
  \node[inner sep = 0.3mm] (t) at (1,1.4) {$\{e,f\}$}; 
  \draw(b) to node[right]{ \ $c$} (r);
  \draw(b) to node[left]{$b$ \ } (l);
  \draw(t) to node[right]{ \ $a$} (r);
  \draw(t) to node[left]{$a+c-b$ \ } (l);   
  \node at (1,-1.4) {(a)};
  \draw[thick](4,0.3)--(6,0.3);
  \filldraw (5,0.3) circle  (2.5pt);
  \node at (5,0) {$e$};
  \node at (5.7,0) {$f$};
  \node at (5,-0.45){$\rho_{\{0,1,2\}}$};
  \node at (5,-1.4) {(b)};
  \end{tikzpicture}
  \caption{(a) \ The rank increases in the polymatroid $\rho_A$ that
    is defined in Theorem \ref{thm:quotkpolyExMin}.  (b) \ The
    polymatroid arising from $\{0,1,2\}$.}
  \label{fig:rankincr}
\end{figure}

\begin{thm}\label{thm:quotkpolyExMin}
  Fix $k\geq 2$. 
  Within the class
  of $k$-polymatroids, $\mathcal{Q}_k$ has $\binom{k+1}{3}$ excluded
  minors, given as follows: for each subset $A=\{a,b,c\}$ of
  $\{0,1,2,\ldots,k\}$ with $a<b<c$, define $\rho_A$ on $\{e,f\}$ by
  $\rho_A(\emptyset)=0$, \ $\rho_A(e) = b$, \ $\rho_A(f) = c$, and
  $\rho_A(\{e,f\}) = a+c$.
\end{thm}

\begin{proof}
  It is easy to see that $\rho_A$ is a $k$-polymatroid and that any
  $k$-polymatroid on a singleton set is in $\mathcal{Q}_k$, so showing
  that $\rho_A \not \in \mathcal{Q}_k$ shows that it is an excluded
  minor for $\mathcal{Q}_k$.  To see that
  $\rho_A\notin \mathcal{Q}_k$, it suffices to show that recurrence
  (\ref{eq:reconr}) gives different values for $r_{M_c}(\{e,f\})$.
  Now $r_{M_c}(f) = r_{M_c}(\emptyset) +1 = 1$, so
  $r_{M_c}(\{e,f\})=r_{M_c}(f) =1$ since $a<c$.  However,
  $r_{M_c}(\{e,f\})=r_{M_c}(e) = r_{M_c}(\emptyset) = 0$ since
  $\rho(\{e,f\}) - \rho(e) = a+c-b<c$ and $b<c$.

  Let $\rho$ be a $k$-polymatroid on $E$ that has no minor that is
  isomorphic to any $\rho_A$.  We must prove that (i) the functions
  $r_{M_i}$ defined by recurrence (\ref{eq:reconr}) are well-defined,
  (ii) each $r_{M_i}$ is the rank function of a matroid, $M_i$, and
  (iii) $M_{i+1}$ is a quotient of $M_i$.

  We show, by induction on $|Z|$, that $r_{M_i}(Z)$ is well-defined.
  This holds if $|Z|\leq 1$, so assume that $|Z|\geq 2$ and that the
  result holds for all sets with fewer elements.  Fix distinct
  elements $e,f\in Z$ and set $Z'=Z-\{e,f\}$.  We claim that taking
  $X=Z-e$ in equation (\ref{eq:reconr}) yields the same value for
  $r_{M_i}(Z)$ as taking $X=Z-f$.  Set
  $$a=\rho(Z)-\rho(Z-e), \,\,
  b=\rho(Z-f)-\rho(Z'), \,\, c=\rho(Z-e)-\rho(Z'), \,\,
  d=\rho(Z)-\rho(Z-f).$$ Thus, $a+c=b+d$.  The claim follows from
  equation (\ref{eq:reconr}) if the two chains from $Z'$ to $Z$ yield
  the same rank increases in some order, that is, if
  $\{a,c\}= \{b,d\}$. 
   We now show that the
  polymatroids $\rho_A$ account for all ways in which the equality can
  fail.  Indeed, if $\{a,c\}\ne \{b,d\}$, then, in particular,
  $b\ne c$, so we may assume, by symmetry, that $b<c$.  Since
  $b\ne a$, submodularity gives $a<b$.  Restricting to $Z$ and
  contracting $Z'$ yields $\rho_{\{a,b,c\}}$ as a minor, contrary to
  our assumption.  Thus, $\{a,c\}= \{b,d\}$, as needed.

  To verify that $r_{M_i}$ is a matroid rank function, we use the
  following axioms: (i) $r(\emptyset) =0$, (ii) if $X\subseteq E$ and
  $e\in E$, then $r(X)\leq r(X\cup e)\leq r(X)+1$, and (iii) for any
  set $X\subseteq E$ and elements $e,f\in E$, if
  $r(X\cup e) = r(X\cup f)=r(X)$, then $r(X\cup \{e,f\})=r(X)$ (local
  submodularity).  Property (i) holds by construction, and equation
  (\ref{eq:reconr}) makes property (ii) evident, so we focus on local
  submodularity.  Since $r_{M_i}(X\cup f) =r_{M_i}(X)$ we must have
  $\rho(X\cup f)-\rho(X)<i$, so submodularity gives
  $$\rho(X\cup \{e,f\})-\rho(X\cup e)\leq \rho(X\cup f)-\rho(X)<i,$$ so
  $r_{M_i}(X\cup \{e,f\})= r_{M_i}(X\cup e)=r_{M_i}(X)$.

  Finally, $M_{i+1}$ is a quotient of $M_i$ since recurrence
  (\ref{eq:reconr}) gives inequality (\ref{eq:quotientrankineq}).
\end{proof}

After switching $e$ and $f$, the $k$-dual of $\rho_{\{a,b,c\}}$ is
$\rho_{\{k-c,k-a-c+b,k-a\}}$.

Note that $\mathcal{Q}_k$ is a proper subclass of $\mathcal{Q}_{k+1}$
and that each excluded minor for $\mathcal{Q}_k$ is an excluded minor
for $\mathcal{Q}_{k+1}$.  This observation gives the next
result.

\begin{cor}\label{cor:quotientexmin}
  The excluded minors for the class
  $\mathcal{Q}=\cup_{k\geq 2}\mathcal{Q}_k$ are the polymatroids
  $\rho_A$ defined above as $A$ ranges over all $3$-element sets of
  nonnegative integers.
\end{cor}

D.~Chun \cite{deb} shows that a $2$-polymatroid $\rho$ has the form
$r_{M\del e}+r_{M/e}$ for some matroid $M$ and element $e\in E(M)$ if
and only if $\rho$ has no minor isomorphic to the matroid $U_{2,2}$ or
the polymatroid $\rho_{\{0,1,2\}}$ in Figure \ref{fig:rankincr} (b).
The polymatroid $\rho_{\{0,1,2\}}$ enters since $M\del e, M/e$ is a
lift/quotient pair, and $U_{2,2}$ enters since
$r(M\del e)- r(M/e)\leq 1$.  More generally, for a fixed positive
integer $t$, the class of $2$-quotient polymatroids $r_{M_1}+r_{M_2}$
with $r(M_1)-r(M_2)\leq t$ is minor-closed.  The final result
identifies the excluded minors for arbitrary $t$.

\begin{thm}\label{thm:QL-pExMinBoundRankDifference}
  A $2$-quotient polymatroid $\rho = r_{M_1}+r_{M_2}$ satisfies
  $r(M_1)-r(M_2)\leq t$ if and only if the uniform matroid
  $U_{t+1,t+1}$ is not a minor.
\end{thm}

\begin{proof}
  While $U_{t+1,t+1}$ is a $2$-quotient polymatroid, it arises only
  from the lift/quotient pair $U_{t+1,t+1}$ and $U_{0,t+1}$, which
  fails the inequality $r(M_1)-r(M_2)\leq t$, while its proper minors
  satisfy that inequality.  Thus, it suffices to show that if, for a
  $2$-quotient polymatroid $\rho = r_{M_1}+r_{M_2}$, we have
  $r(M_1)-r(M_2)>t$, then $U_{t+1,t+1}$ is a minor of $\rho$.  To see
  this, take a basis $B$ of $M_2$, so $B$ is independent in $M_1$;
  extend $B$ to a basis $B'$ of $M_1$; in $\rho$, contracting $B$ and
  deleting $E-B'$ yields the free matroid on $t+1$ or more elements.
\end{proof}

\end{document}